\documentclass[twocolumn]{autart}

\usepackage{mathrsfs}
\usepackage{color}
\usepackage{amsfonts,amssymb}
\usepackage{subfigure}
\usepackage{booktabs}
\usepackage{balance}
\usepackage{graphicx}
\usepackage{float}

\usepackage{enumerate}

\usepackage{algorithm}
\let\classAND\AND
\let\AND\relax
\usepackage{algorithmic}

\let\AND\classAND
\AtBeginEnvironment{algorithmic}{\let\AND\algoAND}

\usepackage{bm}
\usepackage{makecell}
\usepackage{epstopdf}
\usepackage{hyperref}


\usepackage{mathtools}
\usepackage{bbm}
\usepackage{dsfont}
\usepackage{pifont}
\usepackage{cite}
\usepackage{enumitem}

\usepackage{soul}
\usepackage{bbding}
\usepackage{setspace}

\usepackage{courier}
\usepackage{multirow}
\usepackage{makecell}
\usepackage{pifont}

\usepackage[normalem]{ulem}

\makeatletter
\long\def\@maketablecaption#1#2{\@tablecaptionsize
    \global \@minipagefalse
    \hbox to \hsize{\parbox[t]{\hsize}{\centering #1 \\ #2}}}
\makeatother

\newtheorem{remark}{Remark}
\newtheorem{theorem}{Theorem}
\newtheorem{lemma}{Lemma}
\newtheorem{corollary}{Corollary}
\newtheorem{assumption}{Assumption}
\newtheorem{problem}{Problem}
\newtheorem{definition}{Definition}
\newtheorem{proposition}{Proposition}

\DeclareMathOperator{\rank}{rank}
\DeclareMathOperator{\col}{col}
\DeclareMathOperator{\diag}{diag}
\DeclareMathOperator{\vech}{vech}
\DeclareMathOperator{\vecc}{vec}

\begin{document}

\begin{frontmatter}

\title{3DIOC: Direct Data-Driven Inverse Optimal Control \\ for LTI Systems} 

\author{Chendi Qu}\ead{qucd21@sjtu.edu.cn}, 
\author{Jianping He}\ead{jphe@sjtu.edu.cn}, 
\author{and Xiaoming Duan}\ead{xduan@sjtu.edu.cn} 
\address{Department of Automation, Shanghai Jiao Tong University, China}  

\begin{abstract}                      This paper addresses the Direct Data-Driven Inverse Optimal Control (3DIOC) problem for linear time-invariant (LTI) systems under the linear quadratic (LQ) control. Unlike traditional approaches that require system identification, the proposed method learns the underlying objective function directly from measured input-output trajectories. Leveraging the input-output representation of LTI systems via the Fundamental Lemma, we derive a model-free optimality necessary condition (ONC) for the forward LQ problem, which forms the basis for formulating and solving an inverse optimal control problem. We also provide an identifiability condition to ensure the uniqueness of the inverse solution.
While the ONC-based IOC approach is effective in the noise-free case, its performance is not promising when the data is corrupted with noises. We then reformulate the 3DIOC as a bi-level optimization problem, which is solved using iterative gradient descent and offers solution guarantees. Furthermore, we analyze the relationship between the solution sets of the two proposed formulations, providing practical insights into their selection.
The simulation results validate the effectiveness and performance of our proposed methods.

\end{abstract}

\end{frontmatter}

\section{Introduction}
Inverse optimal control (IOC) is to identify the underlying objective function of a control system based on its optimal input/state/output trajectories \cite{ab2020inverse}. As an important branch of learning from demonstration (LfD) methods \cite{ravichandar2020recent}, 
IOC helps understand the high-level intention behind tasks, thereby achieving better generalization performance across environments compared to direct methods (e.g. behavioral cloning).
Recent years have witnessed growing attention in this topic. IOC has been widely applied in autonomous driving \cite{xu2022energy} and human-robot collaborations \cite{mainprice2016goal}, where the learned objective function is used for human intention prediction. 
For imitation learning, robots learn new skills by identifying the objective given expert's demonstration \cite{ruan2023causal}. Other applications include finding the proper cost function for transportation network system, and learning risk preferences of investors, which are extensively studied in Inverse Optimization \cite{chan2025inverse}.

Previous studies on IOC usually assume both the objective function model and the system dynamic equations are known. Algorithms can be classified according to different objective function models. One common option is the linear quadratic form. \cite{priess2014solutions} studies the inverse infinite-time linear quadratic regulator (LQR) control assuming the constant feedback gain matrix is already known, while \cite{yu2021system, qu2023control, zhang2024statistically} identify the weighting matrices for finite-time case. Another general form of the objective function is a weighted sum of specified features \cite{jin2022learning}. \cite{baimukashev2024automated} studies selecting features automatically to relax this prior knowledge assumption.

As for the system dynamic, however, it is difficult to accurately obtain in practice, restricting the application of previous IOC algorithms. Inspired by the forward controller design, data-driven methods can be divided into direct and indirect ways. When considering the indirect data-driven method, System Identification (SI) \cite{ljung2010perspectives} is conducted as a pre-process, while the identification error may affect the subsequent learning performance of the objective function in an unexpected way \cite{byeon2022inverse}. Direct data-driven control based on fundamental lemma has gained increasing interests in these years \cite{verheijen2023handbook,berberich2020data,de2019formulas}, since it requires less system knowledge for implementation than SI and less reliable data than general learning algorithms. It is proved that an LTI system can be equally represented by its collected
input and output trajectories under certain persistent excitation
conditions. Thus, the controller can be designed directly with this input-output representation.

{To the best of our knowledge, model-free IOC remains an issue that has not been fully resolved. Some efforts have been made under the Markov decision process (MDP) framework \cite{ziebart2008maximum,xue2021inverse,garrabe2023convex,garrabe2025convex}.}
{Considering the optimal control problem, \cite{guo2023imitation} tackles a data-driven imitation learning for Linear Quadratic Gaussian control by directly learning the policy gain matrix. It focuses on the infinite horizon case and the gain is static, while our setting is finite horizon and the policy is time-varying.} {\cite{donge2022multiagent,donge2024efficient} investigate the model-free inverse Reinforcement Learning (IRL) problem for multi-agent systems (MAS) defined by Graphical Apprentice Games including two updates.}
Inspired by the deep Koopman representation of the unknown system, \cite{liang2023data} proposes an IOC algorithm to achieve an optimal Koopman operator and the unknown weights estimation together through iterations. Though \cite{liang2023data} tackles a general nonlinear system, it cannot be directly applied to our problem, since it only considers state dynamics but we observe output trajectory instead of the state.
Moreover, these methods are expensive in computation and require substantial observation data due of the dual-loop framework.

Motivated by the above discussion, we develop a novel direct data-driven IOC algorithm. The main challenge lies in how to identify the objective function with less data and computational cost, and the identifiability analysis to the ill-posed inverse problem. {We utilize the fundamental lemma from the behavioral system theory \cite{markovsky2021behavioral}.} With the input-output representation built by collected trajectories, we establish the model-free Karush–Kuhn–Tucker (KKT) condition for LQ control and construct the IOC problem with this optimality necessary condition. We further remove the redundant parameters in the representation to improve the estimation efficiency and robustness. The proposed KKT-based 3DIOC requires a small amount of data including one offline stochastic trajectory and one optimal trajectory, and only needs to solve a quadratic programming. We also tackle the data-driven IOC problem in a bi-level optimization (BLO) formulation. We reveal an inclusion relationship between solutions obtained by KKT-based (or ONC-based) IOC and BLO-IOC.
The main contributions are summarized as follows:
\begin{itemize}
\item We propose a novel direct data-driven IOC approach of LQ control for LTI systems. {Based on the input-output representation, the IOC problem is built using derived model-free KKT conditions as a constraint.} After removing the redundant parameters in the representation, we develop a simplified 3DIOC that requires less computation and observation data.
\item We provide the identifiability condition for the ill-posedness of the inverse problem. Noise-corrupted case is considered with sensitivity analysis under perturbation.
\item We further establish the data-driven IOC in a bi-level optimization formulation. The gradient-based BLO-IOC algorithm converges to stationary points and the bilevel form is proved to possess asymptotical risk consistency with noisy data. A numerical enumeration method is also introduced to achieve the global optimum.
\item Numerical simulations on randomly generated LTI systems demonstrate the computation efficiency and identification performance. Comparing two algorithms, we discover an inclusion relationship between the solution sets in the noise-free case and observe a trade-off in computation cost and numerical stability in noise-corrupted case. This provides a valuable insight for selecting in two IOC approaches.
\end{itemize}

The remainder of the paper is organized as follows. Section \ref{preliminary} introduces the preliminaries and describes the IOC problem formulation. Section \ref{sec-kkt-ioc} provides the main results and algorithms for KKT-based IOC algorithm while Section \ref{sec-blo-ioc} discusses the BLO-IOC. Simulation experiments are shown in Section \ref{sim}, followed by the conclusion in Section \ref{conc}.

\begin{table*}[tb]
    \centering
    \caption{Overview of Proposed Problems and Algorithms}
    \label{tab:compare-algs}
    \begin{tabular}{c|c|c|c|c|c}
    \hline
    \multirow{2}{*}{Formulation} 
    & \multicolumn{3}{c|}{Noise-free + Identifiable} & \multicolumn{2}{c}{Noisy $w^o$} \\ \cline{2-6}
    & Problem & \makecell{Algorithm \\ (Solution set)} & Connection  & Problem & Guarantee  \\ \hline
     \makecell{KKT-based IOC \\ (Sec \ref{sec-kkt-ioc})} & Problem \ref{sim_3Dioc} & Alg. \ref{kkt-ioc-alg} ($\mathcal{I}_1$) & \multirow{3}{*}{ \makecell{$\{\Theta$: Scaling invariant \\ set of true $Q,R \}$ \\ \ding{192} $\mathcal{I}_1 \subseteq \Theta, \mathcal{I}_3 \subseteq \Theta$  \\
     \ding{193} $\mathcal{I}_1 \subseteq \mathcal{I}_2, \mathcal{I}_2 \not\subset \Theta$}}  & Problem \ref{inf_qr_pro} & \makecell{Convexity,\\Sensitivity} \\ 
     \cline{1-3} \cline{5-6}
     \multirow{2}{*}{\makecell{BLO-IOC \\ (Sec \ref{sec-blo-ioc})}} & \multirow{2}{*}{Problem \ref{blo-ioc}}  & Alg. \ref{blo-alg} ($\mathcal{I}_2$)  & & \multirow{2}{*}{\makecell{Problem \eqref{blo-ioc-noisy} \\ (with noisy $\Tilde{u}^p$)}} & \multirow{2}{*}{\makecell{Convergence,\\ Risk consistent}}  \\ \cline{3-3} 
     & & Alg. \ref{enum-alg} ($\mathcal{I}_3$) &  & & \\ \hline
    \end{tabular}
\end{table*}

\noindent \textbf{Notations}: For a matrix $A \in \mathbb{R}^{n \times m}$, $A^{\dagger}$ denotes its Moore-Penrose pseudo inverse. $A(i:j,:)$ (or $A(:,i:j)$) denotes the matrix composed by the $i$ to $j$ rows (or columns). $A\otimes B$ is the Kronecker product between matrices $A$ and $B$. {$\vech(A) \in \mathbb{R}^{\frac{1}{2}n(n+1)}$ represents the half-vectorization of a symmetric matrix $A \in \mathbb{S}^{n}$.} For a series of vectors $v_1,\dots, v_k$, $\col(v_1,\dots, v_k) = (v_1^T,\dots, v_k^T)^T$. $\diag(v)$ represents the diagonal matrix with diagonal elements being the vector $v$. $\diag(A_1,\dots,A_k)$ is the block diagonal matrix composed of matrices $A_1,\dots,A_k$ on the diagonal and $\col(A_1,\dots,A_k) = (A_1^T, \dots, A_k^T)^T$. $w \wedge w'$ denotes the concatenation of two trajectories $w,w'$. $\mathbb{N}(A)$ is the null space of a matrix $A$.

\section{Preliminaries and Problem Formulation}\label{preliminary}
\subsection{Problem Description}
{Consider an LTI system $\mathscr{B}$ that admits a minimal state-space realization:}
\begin{equation}\label{sys}
x_{k+1} = A x_k +Bu_k,~ y_k = C x_k +D u_k,
\end{equation}
where $A \in \mathbb{R}^{n \times n}, B\in \mathbb{R}^{n \times m},C\in \mathbb{R}^{p \times n},D\in \mathbb{R}^{p \times m}$, and $x_k, u_k, y_k$ are respectively the system state, control input and output at time $k$. {The lag of the system $l$ is defined by the smallest integer $l>0$ such that the observability matrix $\mathcal{O}_l:=\col(C, CA,\dots, CA^{l-1})$ has rank $n$. Denote the set of all $L$-length input-output trajectories generated by $\mathscr{B}$ from arbitrary initial states as $\mathscr{B}|_L$.} We have the following assumptions throughout the paper:
\begin{assumption}[Model-free]\label{as-model}
We have no prior knowledge about the matrices $A,B,C,D$ in \eqref{sys}. However, one input-output trajectory 
\begin{equation}
w^d := \{u^d_k, y^d_k\}_{k=0}^{T-1}\in \mathscr{B}|_T,
\end{equation} 
of length $T$ is available.
\end{assumption} 

Assume the system $\mathscr{B}$ in \eqref{sys} is driven by a forward optimal LQ controller described as:
\begin{equation}\label{model-lq}
\mathbf{P}_0: \begin{aligned}
&  \min_{u_{0:N'-1}, y_{0:N'-1}} ~
\sum_{k=0}^{N'-1} y_k^T Q y_k + u_k^T R u_k\\
& ~~~~~~~
s.t. ~~~~~~~~ \eqref{sys}, x_0 = \bar{x},
\end{aligned}
\end{equation}
where $R\in \mathbb{S}^{m}$ and $Q \in \mathbb{S}^{p}$ are positive definite matrices, $\bar{x}$ is the initial state, and $N'$ is the control horizon. We have access to an optimal input-output trajectory denoted by
\begin{equation}\label{wo}
w^o := \{u^o_k,y^o_k\}_{k=0}^{N'-1} \in \mathscr{B}|_{N'}.
\end{equation}
The inverse problem tackled in this paper is described as follow.
\begin{problem}\label{pro-desc}
The direct data-driven IOC problem is to identify the weighting matrices $Q$ and $R$ in the control objective function in problem $\mathbf{P}_0$ directly from collected trajectories $w^d$ and $w^o$.
\end{problem}

\subsection{Preliminary: Fundamental Lemma}
In this section, we introduce the Fundamental Lemma from the behavioral system theory. Behavioral system theory views the LTI dynamical system as a subspace of the signal space in which the system trajectories live. Define a persistently exciting property for the input $u^d$ as follow.

\begin{definition}\label{def-pe}
The input signal $u^d=\{u^d_k\}_{k=0}^{T-1}$ is persistently exciting of order $L$ if the corresponding Hankel matrix
\begin{equation}\nonumber
\mathcal{H}_L(u^d):=\begin{pmatrix}
u^d_0 & \cdots & u^d_{T-L} \\
\vdots & \ddots & \vdots \\
u^d_{L-1} & \cdots & u^d_{T-1}
\end{pmatrix}
\end{equation}
has full row rank, {where $L,T \in \mathbb{Z}_+$ and $T \geq L$.}
\end{definition}
 
\noindent Based on Definition \ref{def-pe}, we have the fundamental lemma.
\begin{lemma}[Fundamental Lemma \cite{willems2005note}]\label{fund_l}
Consider the LTI system $\mathscr{B}$ in \eqref{sys}. If\\
a) system $\mathscr{B}$ is controllable, ~b) $w^d \in \mathscr{B}|_T$,\\
c) the input $u^d$ is persistently exciting of order $L+n$,\\
then we have
\begin{equation}
\mathscr{B}|_L = \mathrm{image}~ \mathcal{H}_L(w^d),
\end{equation}
where 
\[\mathcal{H}_L(w^d) = \begin{pmatrix}
\col(u_0^d, y_0^d) & \cdots & \col(u_{T-L}^d, y_{T-L}^d) \\
\vdots & \ddots & \vdots \\
\col(u_{L-1}^d, y_{L-1}^d) & \cdots & \col(u_{T-1}^d, y_{T-1}^d)
\end{pmatrix}.\]
\end{lemma}
This lemma reveals that under conditions a)-c), any $L$-length trajectory generated by $\mathscr{B}$ can be represented through a linear combination of the columns of $\mathcal{H}_L(w^d)$. The image of $\mathcal{H}_L(w^d)$ is an input-ouput representation of $\mathscr{B}$.
\begin{lemma}[Corollary 19, \cite{markovsky2022identifiability}]\label{gen_fl}
Consider the LTI system $\mathscr{B}$ in \eqref{sys}. For $L \geqslant \max(l,n)$ and the trajectory $w^d$, we have $\mathscr{B}|_L = \mathrm{image}~ \mathcal{H}_L(w^d)$ if and only if
\begin{equation}\label{wd_PE}
\rank(\mathcal{H}_L(w^d))=mL+n.
\end{equation}
\end{lemma}
Lemma \ref{gen_fl} generalizes the fundamental lemma, providing a sufficient and necessary condition without the controllability requirement. Equation \eqref{wd_PE} is generalized persistency of excitation (PE) condition. We utilize Lemma \ref{gen_fl} to establish the input-output representation for $\mathscr{B}$ with $w^d$ in the following section.

\section{Direct Data-Driven IOC}\label{sec-kkt-ioc}
In this section, we first derive the model-free KKT condition for the LQ problem $\mathbf{P}_0$ based on the input-output representation introduced in the preliminary. We build the direct data-driven IOC problem with the obtained optimality necessary condition. Then we simplify the algorithm by removing the redundant parameters in the representation to improve the efficiency.
\subsection{Model-free KKT Condition}
To obtain the model-free KKT condition, we firstly reformulate the problem $\mathbf{P}_0$ as a data-enabled LQ control.

For the collected optimal trajectory $w^o$ in \eqref{wo}, we partition it as an initial trajectory of length $T_{\textup{ini}}$ and the remaining sequence of length $N=N'-T_{\textup{ini}}$, that is
\begin{equation}\label{wo_split}
w^o = w^{\textup{ini}} \wedge \Tilde{w}, ~ w^{\textup{ini}} \in \mathscr{B}|_{T_{\textup{ini}}}, ~ \Tilde{w} \in \mathscr{B}|_{N}.
\end{equation}
For notational convenience, we re-index the element of $w^{\textup{ini}}, \Tilde{w}$ to start from zero, so that $w^{\textup{ini}}=\{u^{\textup{ini}}_k, y^{\textup{ini}}_k\}_{k=0}^{T_{\textup{ini}}-1}$ and $\Tilde{w} = \{\Tilde{u}_k,\Tilde{y}_k\}_{k=0}^{N-1}$, where $u^{\textup{ini}}_k := u^o_k$ for $k=0,\dots,T_{\textup{ini}}-1$ and $\Tilde{u}_k := u^o_{T_{\textup{ini}}+k}$ for $k=0,\dots,N-1$ (also $y^{\textup{ini}}_k, \Tilde{y}_k$).
The following lemma guarantees the uniqueness of the initial state for a $T_{\textup{ini}}$-length trajectory $w^d$ with inputs.
\begin{lemma}[Lemma 1, \cite{markovsky2008data}]\label{lem:uniq}
For a trajectory $w^{\textup{ini}} \in \mathscr{B}|_{T_{\textup{ini}}}$ and an arbitrary $N$-length input sequence $\{u_k\}_{k=0}^{N-1}$, if $T_{\textup{ini}} \geqslant l$, there exists a unique initial state $x_0^{\textup{ini}}$ and unique outputs $\{y_k\}_{k=0}^{N-1}$, such that $w^{\textup{ini}} \wedge w \in \mathscr{B}|_{T_{\textup{ini}}+N}$.
\end{lemma}

For the measured trajectory $w^d$, we define
\[
\begin{pmatrix}
U_p \\ U_f
\end{pmatrix} := \mathcal{H}_{T_{\textup{ini}}+N} (u^d), \begin{pmatrix}
Y_p \\ Y_f
\end{pmatrix} := \mathcal{H}_{T_{\textup{ini}}+N} (y^d),
\]
where $U_p$ is the first $T_{\textup{ini}}$ row of Hankel matrix $\mathcal{H}_{T_{\textup{ini}}+N}(u^d)$ and $U_f$ is the remaining $N$ row (similarly for $Y_p$ and $Y_f$). 
Then we build an offline data matrix based on $w^d$
\begin{equation}\nonumber
H_{T_{\textup{ini}}+N} (w^d) =: \col(U_p, Y_p, U_f, Y_f) \sim \mathcal{H}_{T_{\textup{ini}}+N} (w^d),
\end{equation}
{where $\sim$ denotes a permutation similarity involving
rows.}

\noindent Therefore, according to Lemma \ref{gen_fl}, if $T_{\textup{ini}}+N \geqslant \max(l,n)$, we have the following proposition for $w^d$.
\begin{proposition}\label{prop:pe}
The data matrix $H_{T_{\textup{ini}}+N} (w^d)$ is an input-output representation of system $\mathscr{B}$ if the trajectory $w^d$ satisfies the generalized PE condition \eqref{wd_PE} which is
\begin{equation}\label{as-data}
\begin{aligned}
\rank(H_{T_{\textup{ini}}+N} (w^d)) 
= m(T_{\textup{ini}}+N) + n.
\end{aligned}
\end{equation}
\end{proposition}
{
\begin{remark}[Choice of $T_{ini}$ and $N$]
Proposition \ref{prop:pe} requires $T_{\textup{ini}}\geq l$ and $T_{\textup{ini}}+N\geq \max(l,n)$, while $l,n$ are unknown system quantities under our model-free assumption. Due to $l\leq n$ by Cayley-Hamilton theorem, an upper bound for
$n$ is sufficient. In practice, it is suggested to simply collect a
sufficiently long $w^d$ to exceed the necessary
amount for generalized PE condition. Sweep $N'$ over a small range and select the smallest $N'$ for which $\hat n(\mathscr{B}) = \rank\!\big(H_{N'}(w^d)\big)-mN'$ becomes stable. For $T_{\textup{ini}}$, one may increase it from a candidate value until $\|Y_f N_H\|$ becomes (numerically) negligible where columns of $N_H$ form a basis of $\mathbb{N}(\col(U_p,Y_p,U_f))$, to ensure the uniqueness described in Lemma \ref{lem:uniq}.
\end{remark}}

Based on above discussions, we build the data-enabled LQ problem given $w^d$ and $w^{\textup{ini}}$:
\begin{equation}\nonumber
\mathbf{P}_1: \begin{aligned}
& \min_{g,u,y} ~~~ y^T \mathcal{Q} y + u^T \mathcal{R} u\\
& ~ s.t. ~~~ \col(U_p, Y_p, U_f, Y_f) g = \col(u^{\textup{ini}},y^{\textup{ini}}, u, y),
\end{aligned}
\end{equation}
where $y = \col(y_0, y_1, \dots, y_{N-1}),u = \col(u_0, u_1, \dots, u_{N-1})$ and $\mathcal{Q} = \diag(Q,\dots,Q), \mathcal{R} = \diag(R,\dots,R)$. 
{In the constraint, $g\in \mathbb{R}^{T-T_{ini}-N+1}$ serves as a decision variable.} If \eqref{as-data} is satisfied, then the right side of the constraint is guaranteed to stay in the trajectory set $\mathscr{B}|_{T_{\textup{ini}}+N}$. If $T_{\textup{ini}} \geqslant l$, one solution $u$ combined with $u^{\textup{ini}},y^{\textup{ini}}$ ensures one unique output $y$. 
\begin{proposition}\label{rem-1}
If equation \eqref{as-data} and $T_{\textup{ini}} \geqslant l$ are fulfilled, the $N$-length trajectory $\Tilde{w}$ split from $w^o$ is exactly an optimal solution to problem $\mathbf{P}_1$.
\end{proposition}
Proposition \ref{rem-1} reveals a connection between the optimal solutions of $\mathbf{P}_0$ and $\mathbf{P}_1$. With the formulation of $\mathbf{P}_1$, we establish its model-free KKT conditions as follows.
\begin{lemma}[Model-free KKT]\label{kkt_lem}
The optimal input solution $u^*$ with its corresponding decision variable $g^*$ and output trajectory $y^*$ to $\mathbf{P}_1$ satisfy
\begin{equation}
\begin{aligned}
& \frac{\partial L}{\partial g} = {\lambda^*}^T H_{T_{\textup{ini}}+N} (w^d) = 0,\\
& \frac{\partial L}{\partial u} = 2 \mathcal{R} u^* - \lambda_u^* = 0, \frac{\partial L}{\partial y} = 2 \mathcal{Q} y^* - \lambda_y^* = 0,\\
& \frac{\partial L}{\partial \lambda} = H_{T_{\textup{ini}}+N} (w^d) g^* - \col(u^{\textup{ini}}, y^{\textup{ini}}, u^* , y^*) = 0,
\end{aligned}
\end{equation}
where $L$ denotes the Lagrangian function \begin{equation}\nonumber
\begin{aligned}
L(g,u,y,\lambda) = & y^T \mathcal{Q} y + u^T \mathcal{R} u + \lambda^T ({H}_{T_{\textup{ini}}+N}(w^d) g \\
&- \col(u^{\textup{ini}},y^{\textup{ini}}, u, y)),
\end{aligned}
\end{equation}
and for the costate we denote $\lambda = \col(
\lambda_{\textup{ini}}, \lambda_u, \lambda_y)$.
\end{lemma}


\noindent Denote $w_p = \col(U_p, Y_p)$. From Lemma \ref{kkt_lem}, we derive
\begin{equation}\nonumber
\begin{aligned}
& H_{T_{\textup{ini}}+N}^T(w^d) \lambda^* = (
{U_p^T ~~ Y_p^T} ~~ U_f^T ~~ Y_f^T) \col(
\lambda^*_{\textup{ini}}, \lambda^*_u, \lambda^*_y)\\
& = w_p^T \lambda^*_{\textup{ini}} + U_f^T (2\mathcal{R} u^*) + Y_f^T (2\mathcal{Q} y^*)=0.
\end{aligned}
\end{equation}
According to Proposition \ref{rem-1}, the optimal trajectory $\Tilde{w}$ split from $w^o$ should also satisfy Lemma \ref{kkt_lem}, which is
\begin{equation}\label{kkt_cons}
{w_p^T \lambda_{\textup{ini}} + 2U_f^T \mathcal{R} \Tilde{u} + 2Y_f^T \mathcal{Q} \Tilde{y}=0.}
\end{equation}
{We drop the superscript $*$ here since they are collected data rather than a generic optimizer variable in Lemma 4. The multiplier $\lambda_{\textup{ini}}$ denotes the associated costate consistent with this recorded optimal trajectory.}

\subsection{Vanilla 3DIOC Algorithm}

{
Equation \eqref{kkt_cons} presents an equality relation between the real weighting matrices $Q,R$ along with the optimal inputs $\Tilde{u}$, outputs $\Tilde{y}$ (split from $w^o$). Parameters $w_p, U_f, Y_f$ are calculated from $w^d$.
Therefore, we are able to identify $Q,R$ with equation \eqref{kkt_cons} utilizing collected data $w^o, w^d$ directly.} However, observing \eqref{kkt_cons}, we find the identification of $Q,R$ possesses a scalar ambiguity property and what we can identify is a scaling invariant set.
{\begin{remark}[Scalar Ambiguity]
Given the collected data $w^d$ and optimal trajectory $w^o$, suppose the real weighting matrices $(Q,R)$ satisfy the equation \eqref{kkt_cons} with a costate $\lambda_{\textup{ini}}$, then $(\alpha Q, \alpha R)$ with $\alpha \lambda_{\textup{ini}}$ also satisfy the equation, where $\alpha \in \mathbb{R}_+$ is an arbitrary scalar.
\end{remark}}
\begin{definition}[Scaling Invariant Set]
Considering the IOC problem described in Problem \ref{pro-desc}, we define the parameter set $\Theta= \{(\alpha Q, \alpha R), \alpha \in \mathbb{R}_+ \}$ as its adjoint scaling invariant set, where $Q,R$ are true value.
\end{definition}
\noindent {The inverse problem can be ill-posed. For Problem \ref{pro-desc}, there may exist multiple linearly independent matrix sets of $(Q,R)$ that can generate the same optimal trajectory $w^o$. The following condition ensures identifiability of the scaling invariant set.}

\begin{proposition}\label{iden_th}
Suppose there are two sets of weighting matrix $(Q,R)$ and $(Q',R')$ both satisfying equation \eqref{kkt_cons} with respect to collected data $w^d, w^o$ and equation \eqref{as-data} is fulfilled. We have $Q= \alpha Q', R= \alpha R'$ for a scalar $\alpha \in \mathbb{R}_+$, if
\begin{equation}
\mathrm{dim}(\mathbb{N}(\Tilde{\Psi})) =1,
\end{equation}
{The matrix $\Tilde{\Psi}$ is defined as 
\begin{equation}
\Tilde{\Psi} = \begin{pmatrix}
w_p^T & \Tilde{U}_f & \Tilde{Y}_f\end{pmatrix} \cdot \diag(I,D_m,D_p),
\end{equation}
where $\Tilde{U}_f = \sum_{i=0}^{N-1} \Tilde{u}_i^T \otimes 2U_f^T(im+1:(i+1)m,:)$, $\Tilde{Y}_f =  \sum_{i=0}^{N-1} \Tilde{y}_i^T \otimes 2Y_f^T(ip+1:(i+1)p,:)$ and $D_m, D_p$ are duplication matrices satisfying
\begin{equation}\label{eq:dpdm}
D_m \vech(R) = \vecc(R), D_p \vech(Q) = \vecc(Q).
\end{equation}
Appendix \ref{th1-prf} provides the explicit calculation form of $\Tilde{\Psi}$.}
\end{proposition}
\begin{pf}
    See the proof in Appendix \ref{th1-prf}.
\end{pf}

\noindent The above condition is easy to check in practice by calculating the rank, which should satisfy
\begin{equation}\label{id-cons}
\rank(\Tilde{\Psi}) = (m+p)T_{\textup{ini}}+\frac{1}{2}(m^2+m+p^2+p)-1.
\end{equation}

With the identifiability guarantee, we can obtain a solution included in the scaling invariant set of real weighting matrices by utilizing equation \eqref{kkt_cons} as a constraint. Given $w^d,w^o$, we provide the following data-driven IOC problem:
\begin{problem}\label{deioc}
\begin{equation}
\begin{aligned}
\min_{\lambda_{\textup{ini}}, Q, R} ~ \kappa^2 ~~~~s.t. ~ \eqref{kkt_cons}, I \preceq \mathrm{diag}(Q,R) \preceq \kappa I.
\end{aligned}
\end{equation}
\end{problem}

\noindent Due to the scalar ambiguity, we minimize the condition number of the matrix $\diag(Q,R)$ for a unique solution. Thus, if the collected data is accurate, Problem \ref{deioc} is a feasible LMI problem. Since it is convex with linear constraints, there
exists a unique optimal estimation $\hat{Q}, \hat{R}$.

\subsection{Simplified 3DIOC without Redundant Variables}\label{subsec-sim-kkt}

Notice that the unknown variables in problem $\mathbf{P}_1$ are actually redundant. We do not care the estimation of the Lagrangian multiplier $\lambda_{\textup{ini}}$ and the number of variables in it will increase when $T_{\textup{ini}}$ goes larger, leading to an unnecessary cost in computation. Therefore, in this subsection, we present a simplified KKT-based 3DIOC by removing the redundant variables. 

Substitute optimization variables $y,g$ with $u$ in problem $\mathbf{P}_1$ by converting the constraint into
\begin{equation}\label{y-pseudo}
y = Y_f \begin{pmatrix}
U_p \\ Y_p \\ U_f
\end{pmatrix}^{\dagger} \begin{pmatrix}
u^{\textup{ini}} \\ y^{\textup{ini}} \\ u 
\end{pmatrix} \!=\! K\begin{pmatrix}
u^{\textup{ini}} \\ y^{\textup{ini}} \\ u 
\end{pmatrix}\!=\!
(K_p ~|~ K_f)\begin{pmatrix}
z_{\textup{ini}} \\ u
\end{pmatrix}
\end{equation}
in which case the objective function becomes
\begin{equation}\nonumber
\begin{aligned}
&J(u) =: \begin{pmatrix}
z_{\textup{ini}}^T & u^T
\end{pmatrix} \begin{pmatrix}
K_p^T \\ K_f^T
\end{pmatrix} \mathcal{Q} \begin{pmatrix}
K_p & K_f
\end{pmatrix} \begin{pmatrix}
z_{\textup{ini}} \\ u
\end{pmatrix} + u^T \mathcal{R} u\\
& = z_{\textup{ini}}^T K_p^T \mathcal{Q} K_p z_{\textup{ini}} \!+\! 2z_{\textup{ini}}^T K_p^T \mathcal{Q} K_f u \!+\! u^T (K_f^T \mathcal{Q} K_f \!+\! \mathcal{R}) u,
\end{aligned}
\end{equation}
and problem $\mathbf{P}_1$ turns into an unconstrained problem
\[
\mathbf{P}_2: ~\min_{u} ~~J(u).
\]
\begin{remark}
{
Forward problems $\mathbf{P}_1$ and $\mathbf{P}_2$ are equal with exact data, while when the data is inexact, substitution \eqref{y-pseudo} reflects a regularization $\Vert g \Vert_2$ in $\mathbf{P}_2$, leading to a more robust solution.}
\end{remark}
Therefore, the optimality necessary condition is that the derivative of $J(u)$ with respect to $u$ equals zero. We have
\begin{equation}\nonumber
\frac{\partial J(u)}{\partial u} = 2  (K_f^T \mathcal{Q} K_f+\mathcal{R}) u + 2 K_f^T \mathcal{Q}^T K_p z_{\textup{ini}}.
\end{equation}
Based on this new formulation, we present the simplified model-free KKT condition.
\begin{lemma}[Simplified Model-free KKT]\label{skkt}
Given the collected data $w^d, w^{\textup{ini}}$, the optimal input solution $u^*$ corresponding to $\mathbf{P}_2$ satisfies
\begin{equation}\label{skkt-cons}
(K_f^T \mathcal{Q} K_f+\mathcal{R}) u^* + K_f^T \mathcal{Q}^T K_p z_{\textup{ini}} = 0.
\end{equation}
\end{lemma}
Similarly, the optimal trajectory $\Tilde{w}$ split from $w^o$ also satisfy equation \eqref{skkt-cons}, which is
\begin{equation}\label{skkt-eq}
(K_f^T \mathcal{Q} K_f+\mathcal{R}) \Tilde{u} + K_f^T \mathcal{Q}^T K_p z_{\textup{ini}} = 0.
\end{equation}
With equation \eqref{skkt-eq}, we provide a new identifiability condition for $Q,R$ as follows.
\begin{theorem}[Identifiability]\label{s-iden_th}
Suppose there are two weighting matrix sets $(Q,R)$ and $(Q',R')$ both satisfy equation \eqref{skkt-eq} with respect to collected data $w^d, w^o$ ensuring equation \eqref{as-data}. We have $Q= \alpha Q', R= \alpha R'$ for a scalar $\alpha \in \mathbb{R}_+$, if
\begin{equation}
\mathrm{dim}(\mathbb{N}(\Tilde{\Phi}))=1.
\end{equation}
{The matrix $\Tilde{\Phi}$ is defined as
\begin{equation}\label{eq:tilde-phi}
\Tilde{\Phi} = \begin{pmatrix}
\Phi_y & \Phi_u
\end{pmatrix} \diag(D_p, D_m)
\end{equation}
where $\Phi_y = \sum_{i=0}^{N-1} (\Tilde{y}^{es}_i)^T \otimes K_f^T(ip+1:(i+1)p,:)$ and $\Phi_u = \sum_{i=0}^{N-1} (\Tilde{u}_i)^T \otimes I_{mN}(im+1:(i+1)m,:)$ and $\Tilde{y}^{es} = K_f \Tilde{u} + K_p z_{\textup{ini}}$. $D_p, D_m$ are defined as \eqref{eq:dpdm}.
Appendix \ref{thm2-prf} provides the full explicit calculation form of $\Tilde{\Phi}$.}
\end{theorem}
\begin{pf}
See the proof in Appendix \ref{thm2-prf}.
\end{pf}
Similarly, the identifiability condition \eqref{s-id-cons} is easy to check given the explicit form of $\Tilde{\Phi}$, whose rank should satisfy:
\begin{equation}\label{s-id-cons}
\rank(\Tilde{\Phi}) = \frac{1}{2}(m^2+m+p^2+p)-1.
\end{equation}

Note that matrix $\Tilde{\Phi}$ is of $mN \times \frac{1}{2}(m^2+m+p^2+p)$. {We can further obtain an unidentifiable condition for our inverse problem (Problem \ref{pro-desc}) as follows. Corollary \ref{coro} provides a method to protect the system's intention by designing the horizon length.}
\begin{corollary}[Intention Protection]\label{coro}
The real weighting matrices $Q,R$ (its scaling invariant set) cannot be identified if the horizon $N = N'-T_{\textup{ini}}$ is set to be
\[
N < \frac{1}{2m}(m^2+m+p^2+p)-\frac{1}{m}.
\]
\end{corollary}
\begin{pf}
If $\rank(\Tilde{\Phi}) =mN < \frac{1}{2}(m^2+m+p^2+p)-1$, then there exist multiple linearly independent solutions for $Q,R$ satisfying equation \eqref{skkt-eq}.
\end{pf}

\begin{remark}
Corollary \ref{coro} reveals that the identifiability is related to the richness of collected data. Actually, based on Theorem \ref{s-iden_th}, the unidentifiability condition is described as $\mathrm{dim}(\mathbb{N}(\Tilde{\Phi})) > 1$. In that case, we can still obtain an estimation, which generates the same trajectory under given $N$ with real $Q,R$, while when $N$ becomes larger, the difference in optimal trajectories will appear.
\end{remark}

Holding the simplified KKT condition \eqref{skkt-eq} as constraint, 3DIOC problem is now formulated as
\begin{problem}[KKT-based 3DIOC]\label{sim_3Dioc}
\begin{equation}
\begin{aligned}
\min_{Q, R} ~ \kappa^2 ~~~~ ~s.t.~ \eqref{skkt-eq}, I \preceq \mathrm{diag}(Q,R) \preceq \kappa I.
\end{aligned}
\end{equation}
\end{problem}

In Problem \ref{sim_3Dioc}, we only need the optimal input $\Tilde{u}$ instead both $\Tilde{y},\Tilde{u}$ compared with the former formulations in Problem \ref{deioc}. Besides, since the number of optimization variables is reduced by $(m+p)T_{\textup{ini}}$, the solution is more efficient and robust. See the KKT-based IOC algorithm in Algorithm \ref{kkt-ioc-alg}.

\begin{algorithm}
\caption{KKT-based 3DIOC Algorithm}
\begin{algorithmic}[1]\label{kkt-ioc-alg} 
\REQUIRE 
    One input-output trajectory $w^d$ satisfying PE condition \eqref{wd_PE}; 
    One optimal input-output trajectory $w^o$; 
    The initial trajectory length, $T_{\textup{ini}}$;
\ENSURE Estimated weighting matrices $\hat{Q},\hat{R}$;
\STATE Calculate the Hankel matrix and the corresponding $U_p,Y_p,U_f,Y_f$ with $w^d$;
\STATE Split $w^o$ by $T_{\textup{ini}}$ to get $w^{\textup{ini}}$ and $\Tilde{w}$ as \eqref{wo_split};
\STATE Calculate $K_p,K_f$ with \eqref{y-pseudo};
\STATE Solve Problem \ref{sim_3Dioc} with $K_p,K_f,\Tilde{u},z_{\textup{ini}}$;
\RETURN The optimal estimation $\hat{Q},\hat{R}$.
\end{algorithmic}
\end{algorithm}

\noindent Note that Problem \ref{sim_3Dioc} has a unique optimal solution when the identifiability condition \eqref{s-id-cons} is satisfied. 
{However, in practice, the data collected is usually disturbed by observation noises. Perturbed data as $\Tilde{u}^p = \Tilde{u} + \Delta \Tilde{u}, \Tilde{y}^p = \Tilde{y} +\Delta \Tilde{y}$ leads to the full rank of $\Tilde{\Phi}$, in which case equation \eqref{skkt-eq} only has a trivial zero solution. To deal with this infeasible case, we construct the following Problem \ref{inf_qr_pro}.
{\begin{problem}[Noise-corrupted Case]\label{inf_qr_pro}
\begin{equation}\label{inf_qr}
\begin{aligned}
& \min_{Q, R} ~\Vert (K_f^T \mathcal{Q} K_f+\mathcal{R}) \Tilde{u}^p + K_f^T \mathcal{Q}^T K_p z_{\textup{ini}} \Vert_2^2 \\
& ~~~s.t.~~~ Q \succeq \varepsilon I,~ R \succeq \varepsilon I,~ \mathrm{tr}(Q)+\mathrm{tr}(R)=1,
\end{aligned}
\end{equation}
\end{problem}}
}

{
Problem \ref{inf_qr_pro} is a convex optimization problem (a convex quadratic objective over a spectrahedron), hence it admits a global optimum and can be solved reliably by standard SDP solvers.
To retain analytical insight on \emph{noise sensitivity}, we additionally consider a relaxed \emph{spectral estimator} obtained by dropping the positive-definiteness and trace-normalization constraints, which yields a Rayleigh-quotient problem. The following corollary characterizes the sensitivity of this spectral estimator and shows its dependence on the spectral gap of $\Tilde{\Phi}^p$.
\begin{corollary}[Sensitivity Analysis]\label{per-ana}
Consider the relaxed estimator
\[
\rho(\Tilde{\Phi}^p) := \arg \min_{\theta} \Vert \Tilde{\Phi}^p \cdot \theta \Vert_2^2 ~~~ \text{s.t.} ~ \Vert \theta \Vert_2^2 =1,
\]
{where $\theta \in \mathbb{R}^{\frac{m+m^2+p+p^2}{2}}$ represents $\col(\vech(Q),\vech(R))$. $\Tilde{\Phi}^p$ is defined in the same form as $\Tilde{\Phi}$ in \eqref{eq:tilde-phi} by substituting $\Tilde{u}$ with $\Tilde{u}^p$.}
If there exists a perturbation $\Delta$ on the coefficient matrix $\Tilde{\Phi}^p$, the solution is bounded as
\begin{equation}
\Vert \rho(\Tilde{\Phi}^p) - \rho(\Tilde{\Phi}^p+\Delta) \Vert \leq \frac{\Vert \Delta^T \Tilde{\Phi}^p+ (\Tilde{\Phi}^p)^T \Delta \Vert}{|\sigma_1- \sigma_2|},
\end{equation}
where $\sigma_1, \sigma_2$ are the smallest and second smallest \emph{nonzero} singular values of $\Tilde{\Phi}^p$, respectively.
\end{corollary}}
\begin{pf}
See the proof in Appendix \ref{thm3-prf}.
\end{pf}

As a conclusion of KKT-based 3DIOC, Fig. \ref{flow} illustrates the connection between above proposed problems. In Section \ref{sec-blo-ioc}, we will further explore the inverse problem in a bi-level optimization formulation. 
\begin{figure}[htb]
\centering
\includegraphics[width=0.47\textwidth]{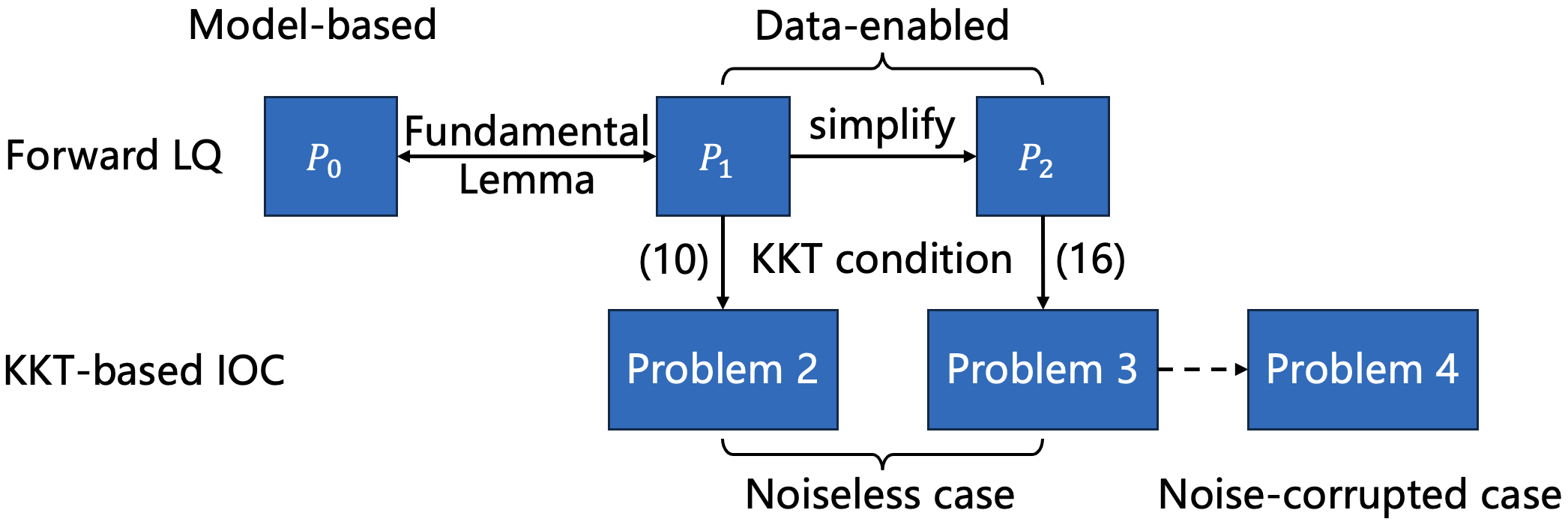}
\caption{{The flow chart of KKT-based 3DIOC problems.}}
\label{flow}
\end{figure}

\subsection{Special Case for LQR-IOC}
This subsection shows $\mathbf{P}_0$ is a general formulation and LQR-IOC is a special case of ours. If we set $C=I, D=0$ in the LTI system $\mathscr{B}$, we can build a discrete-time finite-horizon LQR control problem described as
\begin{equation}
\begin{aligned}
& \min_{u_{0:N'-1}, y_{0:N'}} ~ \sum_{k=0}^{N'-1} (y_k^T Q y_k + u_k^T R u_k) + y_{N'}^T H y_{N'}\\
& ~~~~~ s.t. ~~~~~~ x_{k+1}=Ax_k+Bu_k, y_k=x_k,  x_0 = \bar{x}.
\end{aligned}
\end{equation}
For this case, we collect an output trajectory $y^d$ of length $T+1$ generated by the $T$-length input signal $u^d$. Then we have
\[
\begin{pmatrix}
Y_p \\ \bar{Y}_f
\end{pmatrix} := \mathcal{H}_{T_{\textup{ini}}+1+N} (y^d), \bar{y}=\begin{pmatrix}
y \\ y_N
\end{pmatrix}.
\]
Substitute the $Y_f, y$ in problem $\mathbf{P}_1$ with $\bar{Y}_f, \bar{y}$ and the weighting matrix $\mathcal{Q}$ with $\diag(\mathcal{Q},H)$. The remaining analysis is similar to the previous subsection.

\section{3DIOC in Bi-level Formulation}\label{sec-blo-ioc}
In the previous section, we formulated the IOC problem based on KKT conditions and obtained the estimation by solving a single-level SDP or QP problem. While the KKT-based IOC approach is well-suited for the noiseless case, we encounter difficulties when analyzing the noise-corrupted case. In this section, we propose a bi-level optimization formulation for IOC as a better option for dealing with noises and we further reveal its connection to the previous algorithms.

\subsection{Bi-level Optimization Algorithm}
Consider the following bi-level optimization problem.
\begin{problem}[Bi-level 3DIOC]\label{blo-ioc}
\begin{equation}\label{blo-ioc-eq}
\begin{aligned}
& \min_{\theta \in \mathcal{D}} ~ \mathcal{S}(\theta): l(\theta; u^*_{\theta})=\Vert \Tilde{u} - u^*_{\theta} \Vert_2^2 \\
& ~s.t. ~~ u^*_{\theta} = \arg \min_{u} J(u;\theta),
\end{aligned}
\end{equation}
where the optimization variable set
\[
\mathcal{D} = \{\theta : \theta = \col(\vech(Q), \vech(R)), \underline{\varepsilon}I \preceq Q,R \preceq \bar{\varepsilon}I\}
\]
with constants 
$0<\underline{\varepsilon}<\bar{\varepsilon}< \infty$.
\end{problem}
\noindent We optimize $\theta = \col(\vech(Q), \vech(R))$ in $\mathcal{D}$, where $Q,R$ are constrained to facilitate the subsequent convergence analysis. The upper level objective is to minimize a loss function $l(\cdot)$ measuring the distance between the expert's behavior $\Tilde{u}$ and the learned behavior $u_{\theta}^*$ under a weighting parameter $\theta$. Here we specify $l(\cdot)$ as the L2-norm. The lower level problem is the unconstrained forward LQ control $\mathbf{P}_2$ with a fixed $\theta$. This BLO-IOC formulation is quite intuitive, where the overall objective is to obtain an estimation  $\hat{\theta}$ such that the trajectory it generates is closest to the expert's. This aligns the goal of imitation learning.

To solve Problem \ref{blo-ioc}, we have the following iterative gradient descent algorithm.

\begin{algorithm}[h]
\caption{BLO-3DIOC Algorithm}
\begin{algorithmic}[1]\label{blo-alg}
\REQUIRE 
    Collected input-output trajectories, $w^d,w^o$; 
    The learning rates of outer and inner problems, $\alpha, \beta$; 
    Initial guess to weighting parameter, $\theta_0$; Initial input, $u_{0,0}$; 
    The tolerance sequence, $\{\epsilon_k\}_k$; The constraint set, $\mathcal{D}$;
\ENSURE 
    The estimated weighting parameter $\hat{\theta}$ with
its corresponding control input;
\STATE Calculate $K_p,K_f,\Tilde{u},z_{\textup{ini}}$;
\FOR{Iteration $k=0$ to $\mathtt{max\_iter}$}
    \STATE Solve the inner optimization problem:
    \STATE Let $i = 0$
    \WHILE{$\Vert \nabla_u J(u_{i,k};\theta_k) \Vert \geq \epsilon_k$}
        \STATE $u_{i+1,k} = u_{i,k} - \beta \nabla_u J(u_{i,k};\theta_k)$
        \STATE $i = i+1$
    \ENDWHILE
    \STATE We have the approximate solution to the inner
problem $\bar{u}_{\theta_k} = u_{i,k}$;
\STATE Reset the initial $u_{0,k+1} = u_{i,k}$;
    \STATE Update the outer optimization problem:
    \STATE $\theta_{{new}} = \theta_k - \alpha (\Tilde{u} - \bar{u}_{\theta_k}) \cdot$ \\
$~~~~~~~~~~~~~~[\nabla_{uu}J(\bar{u}_{\theta_k};\theta_k)]^{-1}\nabla_{u\theta}J(\bar{u}_{\theta_k};\theta_k)$
    \STATE Projection on the PSD cone: $\theta_{k+1} \gets \mathcal{P}_{\mathcal{D}}(\theta_{{new}})$
\ENDFOR
\RETURN $\hat{\theta} = \theta_{\mathtt{max\_iter}}$
\end{algorithmic}
\end{algorithm}

\begin{remark}
The lower level optimization problems in Problem \ref{blo-ioc} is strongly convex with respect to $u$ and the Hessian matrix $\nabla_{u u} J$ is invertible.
During the gradient descent, the upper level optimization update may cause the optimization variable $\theta_{new} \not\in \mathcal{D}$. Thus we need to reconstruct $Q,R$ from $\theta_{new}$ and project them onto the bounded positive semi-definite (PSD) cone $\{\underline{\varepsilon}I \preceq Q,R \preceq \bar{\varepsilon}I\}$ to obtain $\theta_{k+1}$. Since a bounded PSD cone is convex, $\mathcal{D}$ is also a convex set. Projections onto convex sets are firmly nonexpansive operators \cite{parikh2014proximal}.
\end{remark}

\begin{theorem}[Convergence]\label{blo-conv}
If the tolerance sequence is summable, that is $\epsilon_k$ is positive and satisfies $\sum_{k=0}^{\infty} \epsilon_k < \infty$, then we have
\begin{equation}
\lim_{k \rightarrow \infty} \Vert \nabla \mathcal{S}(\theta_k) \Vert = 0.
\end{equation}
Then Algorithm \ref{blo-alg} converges to a stationary point.
\end{theorem}
\begin{pf}
See the proof in Appendix \ref{th-3-pf}.
\end{pf}

Now we consider the noisy case. We examine the asymptotic property of the estimator with respect to the amount of collected data. {We use the term \emph{risk} in the standard statistical learning sense, namely the expected prediction error under the data-generating distribution. Assume $\mathcal{T}$ i.i.d. optimal trajectories $\{\Tilde{u}^p(1), \dots, \Tilde{u}^p(\mathcal{T})\}$ perturbed by observation noises are distributed from a fixed distribution. By minimizing the empirical risk, the noise-corrupted BLO-IOC problem is built as}
\begin{equation}\label{blo-ioc-noisy}
\begin{aligned}
& \min_{\theta \in \mathcal{D}} ~ \mathcal{S}_\mathcal{T}(\theta): l(\theta; u^*_{\theta})=\frac{1}{\mathcal{T}} \sum_{i=1}^\mathcal{T} \Vert \Tilde{u}^p(i) - u^*_{\theta} \Vert_2^2 \\
& ~s.t. ~~ u^*_{\theta} = \arg \min_{u} J(u;\theta).
\end{aligned}
\end{equation}
Denote $\hat{\theta}_\mathcal{T}$ as the solution to Problem \eqref{blo-ioc-noisy}. We will show that when the data amount goes to infinity, $\hat{\theta}_\mathcal{T}$ asymptotically provides the best predictions possible. 

\begin{corollary}[Risk Consistency]\label{risk-consis}
For all $i=1,\dots,\mathcal{T}$, assume $\mathbb{E}(\Vert \Tilde{u}^p(i) \Vert^2) < +\infty$. The estimate $\hat{\theta}_\mathcal{T}$ is risk consistent, meaning as $\mathcal{T} \rightarrow \infty$, we have
\begin{equation}
\mathcal{S}(\hat{\theta}_\mathcal{T}) \xrightarrow{P} \min \{ \mathcal{S}(\theta)| \theta \in \mathcal{D}\}.
\end{equation}
With a slight abuse of notation, here we denote $\mathcal{S}(\theta) = \mathbb{E}(\Vert \Tilde{u}^p - u^*_{\theta} \Vert_2^2)$ with $u^*_{\theta} = \arg \min_{u} J(u;\theta)$.
\end{corollary}

\begin{pf}
See the proof in Appendix \ref{co-2-pf}.
\end{pf}

\subsection{Numerical Approach to Achieve the Global Optimum}

Due to the possible non-convexity of $\mathcal{S}(\theta)$ with respect to $\theta$, Algorithm \ref{blo-alg} is a first-order gradient descent framework for bi-level problem that only achieves the local optimum. The estimation performance is highly related to the option of the initial guess $\theta_0$. Inspired by \cite{aswani2018inverse}, we introduce an enumeration algorithm to obtain the global optimal solution based on discretizing $\mathcal{D}$ with a $\delta-$net. The $\delta-$net $\mathcal{T}(\delta)$ is a finite set such that
\[
\max_{\theta \in \mathcal{D}} \min_{t \in \mathcal{T}(\delta)} \Vert t-\theta \Vert \leq \delta.
\]
Then we only need to calculate $\mathcal{S}(\theta)$ over different fixed values of $\theta$ in $\mathcal{T}(\delta)$. This translates the continuous optimization into an enumerative grid search.

\begin{algorithm}[htb]
\caption{Enumeration BLO-3DIOC}
\begin{algorithmic}[1]\label{enum-alg}
\REQUIRE 
    A fixed interval $\delta > 0$; 
    Collected input-output trajectories $w^d, w^o$;
\ENSURE Estimated weighting parameter $\hat{\theta}$;
\STATE Construct the $\delta$-net $\mathcal{T}(\delta) \text{ of } \mathcal{D}$
\FOR{each $\theta \in \mathcal{T}(\delta)$}
    \STATE Compute $\mathcal{S}(\theta)$ by solving Problem~\ref{blo-ioc} with fixed $\theta$
\ENDFOR
\STATE Select optimal: $\hat{\theta} = \arg\min\, \{\mathcal{S}(\theta) \mid \theta \in \mathcal{T}(\delta)\}$
\end{algorithmic}
\end{algorithm}
See the detailed method in Algorithm \ref{enum-alg}.
The enumeration algorithm is applicable to relatively modest $m,p$.

\subsection{Discussion}\label{subsec-dis}
This subsection aims to discuss the relationship between the KKT-based IOC in Section \ref{sec-kkt-ioc} and BLO-IOC in Section \ref{sec-blo-ioc}, offering insights on choosing the proper algorithm in practice. 
We evaluate this in two situations. 

\textbf{Case 1: The problem is noiseless and identifiable.} Given the observed input-output trajectories $w^d, w^o$, we already know the noiseless KKT-based IOC returns the scaling invariant set of the true matrices $Q,R$. By solving Problem \ref{sim_3Dioc}, we obtain the optimal estimation that satisfies equation~\eqref{skkt-eq}.
For Problem \ref{blo-ioc}, in this case the optimal objective function value should be $0$, thus the problem transforms into finding $\theta$ such that $u^*_{\theta}=\Tilde{u}$, which exactly equals Problem \ref{sim_3Dioc}.

However, though the two formulations are equivalent, the gradient descent method such as Algorithm \ref{blo-alg} can only achieve the local minimum of Problem \ref{blo-ioc}.
Lemma \ref{lem-blo-opt} gives the second-order optimality condition of Problem \ref{blo-ioc}, describing its local minimum set.
\begin{lemma}\label{lem-blo-opt}
Assume the gradient descent method to Problem \ref{blo-ioc} converges.
The obtained stationary point $\theta^*$ (also $Q^*,R^*$) satisfies:
\begin{equation}\label{blo-opt-con1}
\nabla \mathcal{S}(\theta^*) =(\Tilde{u} - u^*_{{\theta}^*})^T (K_f^T {\mathcal{Q}^*} K_f+{\mathcal{R}^*})^{-1} \Tilde{\Phi}(u^*_{{\theta}^*}) = 0,
\end{equation}
and 
\begin{equation}\label{blo-opt-con2}
\nabla^2 \mathcal{S}(\theta^*) \succ 0,
\end{equation}
where $\Tilde{\Phi}(u^*_{{\theta}^*})$ is defined by substituting $\Tilde{u}$ in the calculation of $\Tilde{\Phi}$ in Appendix \ref{thm2-prf} by $u^*_{{\theta}^*}$.
\end{lemma}
\begin{pf}
Gradient descent method almost surely converges to second‐order stationary points. This directly comes from the second-order necessary and sufficient condition. See the derivation of \eqref{blo-opt-con1} in Appendix \ref{lem-6-pf}.
\end{pf}
Holding equations \eqref{skkt-eq} and Lemma \ref{lem-blo-opt}, we observe the following inclusion relationship.
\begin{proposition}[Connection between Algorithms]\label{connect}
Assume algorithms achieve exact convergence. The solution to Problem \ref{sim_3Dioc} by Algorithm \ref{kkt-ioc-alg} is included in the solution set obtained by Algorithm \ref{blo-alg} with $\underline{\varepsilon} \leq 1$.
\end{proposition}
\begin{pf}
See the proof in Appendix \ref{pro-4-pf}.
\end{pf}
Proposition \ref{connect} is consistent with our intuition since the gradient descent algorithm finds all stationary points. If Problem \ref{blo-ioc} is strictly convex, two algorithms both return solutions included in the scaling invariant set $\Theta$ theoretically.
Also naturally, since the enumeration method described in Algorithm \ref{enum-alg} achieves the global optimum with a proper $\mathcal{T}(\delta)$, its solution is in $\Theta$.

\begin{figure}[htb]
    \centering
    \includegraphics[width=1\linewidth]{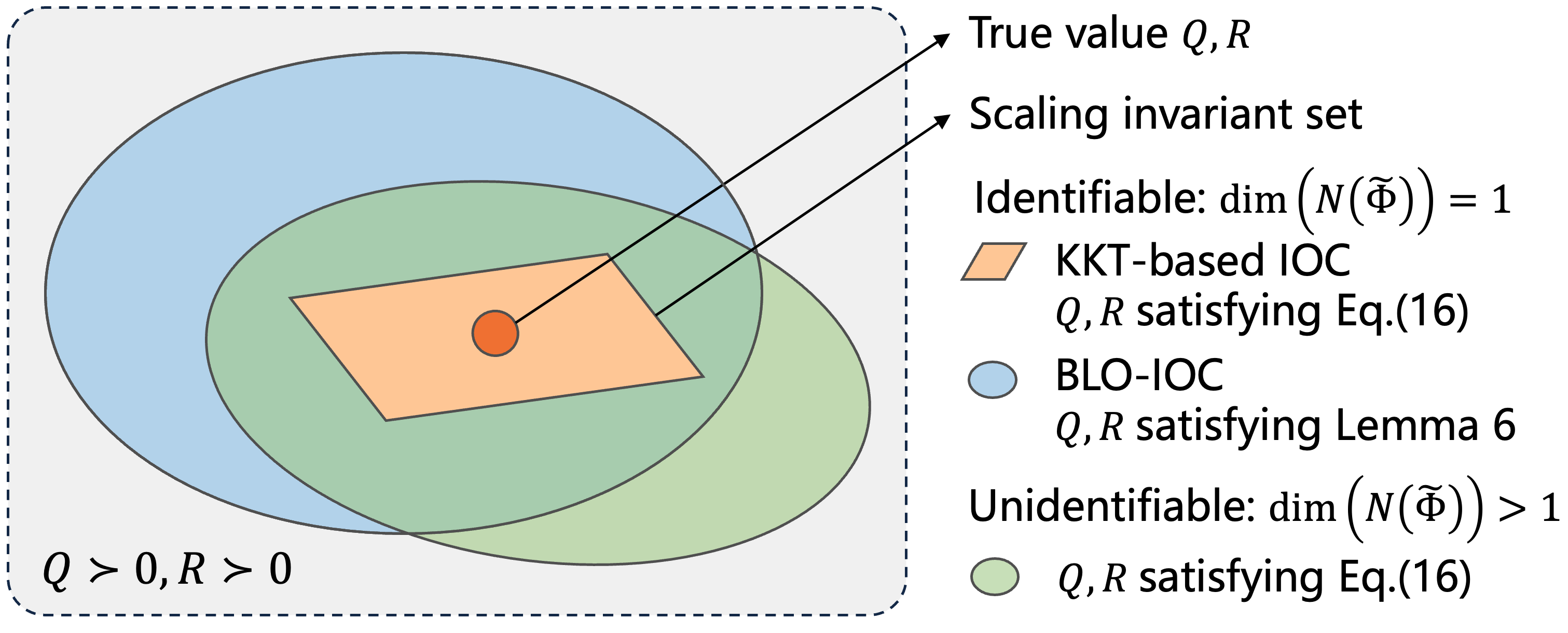}
    \caption{An illustration to the relationship between solution sets in noiseless case. Consider all the potential estimation candidates for $Q,R$, i.e., positive definite matrices, as the gray box. The red dot represents the true value.  
    Green circle shows all the estimates satisfying \eqref{skkt-eq}.
    These parameters generate the same $\Tilde{u}$ in the forward control given $N$. Specifically, when the identifiability condition holds, solution to KKT-based IOC is a scaling invariant set of $Q,R$, illustrated as the parallelogram. Solution set obtained by solving BLO-IOC through gradient descent method is the light blue circle. It returns the local optimum estimates over a risk minimization. }
    \label{fig:connect}
\end{figure}

\textbf{Case 2: The problem involves observation noises.}
When the collected data is disturbed by observation noises, the problem is almost surely unidentifiable. KKT-based IOC is formulated as Problem \ref{inf_qr_pro}. We try to explore the connection between the objective function in Problem \ref{inf_qr_pro} and Problem \ref{blo-ioc} (directly substituting $\Tilde{u}$ with perturbed data $\Tilde{u}^p$). Assuming the inner problem in Problem \ref{blo-ioc} is optimized to optimum, we rewrite $u^*_{\theta}$ as its explicit solution form. The deviation $\Vert \Tilde{u}^p - u^*_{\theta} \Vert_2^2$ now turns into
\begin{equation}\nonumber
\begin{aligned}
& \Vert \Tilde{u}^p +(K_f^T \mathcal{Q} K_f+\mathcal{R})^{-1} (K_f^T \mathcal{Q}^T K_p z_{\textup{ini}})  \Vert_2^2 =\\
&\Vert (K_f^T \mathcal{Q} K_f+\mathcal{R})^{-1} ((K_f^T \mathcal{Q} K_f+\mathcal{R}) \Tilde{u}^p + K_f^T \mathcal{Q}^T K_p z_{\textup{ini}})  \Vert_2^2 \\
\end{aligned}
\end{equation}
which weighs the residual
\[
\Vert (K_f^T \mathcal{Q} K_f+\mathcal{R}) \Tilde{u}^p + K_f^T \mathcal{Q}^T K_p z_{\textup{ini}} \Vert_2^2
\]
in Problem \ref{inf_qr_pro} by the inverse $(K_f^T \mathcal{Q} K_f+\mathcal{R})^{-1}$. Observing these two objectives in different formulations, they are generally not equivalent. 
Overall, Problem \ref{inf_qr_pro} is a non-convex optimization problem with an equality-sphere constraint. We cannot come up with solution guarantees when using a black-box solver in simulation. Problem \ref{blo-ioc} formulated through minimizing the discrepancy to expert’s behavior, by contrast, enjoys convergence guarantees to a local minimum. Moreover, if we turn to tackle the asymptotic problem \eqref{blo-ioc-noisy}, the estimator $\hat{\theta}_\mathcal{T}$ is proved to be risk consistent. 

Therefore, based on the foregoing discussion, we recommend using Algorithm \ref{kkt-ioc-alg} when the problem is noiseless, and adopting the bi-level formulation with Algorithm \ref{blo-alg} when observation noise is present, which is particularly suitable for robot's imitation tasks. 

{To further summarize the practical trade-offs among the three IOC solvers (KKT-based IOC, BLO-IOC, and the enumeration-based approach), we provide a concise comparison in Table~\ref{tab:compare-algs} under two regimes: the noise-free and noise-corrupted settings.
In the noise-free identifiable regime, the true weighting parameters are identifiable up to a scaling ambiguity, hence they form a scaling-invariant set $\Theta$.
The table highlights the set-theoretic connections among the solution sets returned by different algorithms as shown in Fig. \ref{fig:connect}: Algorithm~1 and 3 returns a solution set $\mathcal I_1$ that is contained in $\Theta$; in contrast, Algorithm~2 produces a broader set $\mathcal I_2$ with $\mathcal I_1\subseteq \mathcal I_2$ but without the guarantee $\mathcal I_2\subseteq\Theta$.
When the expert trajectory is corrupted by noise, we show that KKT-based IOC admits a sensitivity analysis, while BLO-IOC provides convergence and risk-consistency guarantees.
}



\section{Simulations}\label{sim}
\subsection{Data Preparation}
{Consider a time-invariant linear system \eqref{sys} with randomly generated matrices:
\begin{align*}
& A = \begin{pmatrix}
0.8147 &  0.1270  &  0.6324\\
0.9058  &  0.9134  &  0.0975\\
0.1270  &  0.6324  &  0.2785
\end{pmatrix}, B = \begin{pmatrix}
0.5469  &  0.9575\\
0.9157  &  0.9649\\
0.7577  &  0.1576
\end{pmatrix},\\
& C = \begin{pmatrix}
0.8003  &  0.4218  &  0.9157\\
0.1419  &  0.9157  &  0.7922\\
0.6557  &  0.7922  &  0.9595
\end{pmatrix}, D = \begin{pmatrix}
0.0357  &  0.8491\\
0.8491  &  0.9340\\
0.9340  &  0.6787
\end{pmatrix}.
\end{align*}
and $n=p=3,m=2$. Set the initial state as a random vector $\bar{x} \in \mathbb{R}^3$ and generate a sequence of inputs to collect input-output trajectories $w^d$ with the dynamic for a length $T=50$. Ensure the collected trajectory $w^d$ satisfies the PE condition. We build the forward LQ problem with weighting matrices
\[
Q = \begin{pmatrix}
1 & 0.2 & 0\\
0.2 & 1.5 & 0.2\\
0 & 0.2 & 0.8
\end{pmatrix}, R = \begin{pmatrix}
0.4 & 0.2\\
0.2 & 0.8
\end{pmatrix}
\]
and horizon $N'=11$.}
Suppose the system is driven by the optimal policy and we collect an optimal input-output trajectory $w^o$. {We construct the Hankel matrix $H_{T_{\textup{ini}}+N}(w^d)$ with the initial length $T_{\textup{ini}}$ from $1$ to $7$ (the left length $N$ from $10$ to $4$). The identifiability is checked through Theorem \ref{iden_th} and \ref{s-iden_th}. When $N < 4$, the inverse problem is unidentifiable according to Corollary \ref{coro}.} 

\subsection{Simulation for KKT-based IOC}
We firstly conduct the KKT-based IOC as Algorithm \ref{kkt-ioc-alg}. The algorithm is run on MacBook Air with Apple M1 eight-core CPU and 16 GB of RAM.
Solve the noiseless problem with semi-definite programming (SDP) solver SeDuMi 
and infeasible case with BMIBNB in MATLAB YALMIP. 
The estimation errors with noiseless data are shown in Fig. \ref{err-compare}. We use the Frobenius norm to measure the estimation error of the weighting matrices in the objective function:
\[err = \inf_{\alpha >0} \frac{\Vert \diag(
\hat{Q}, \hat{R})
\cdot \alpha - \diag(Q, R) \Vert_F}{\Vert \diag(Q, R)\Vert_F}.\]
From Fig. \ref{err-compare}, we can find that the estimation errors remain overall low. {Since Problem \ref{deioc} needs to optimize more variables and computation contains numerical noises, simplified 3DIOC problem achieves a more robust and accurate estimation. We also observe that as $T_{\textup{ini}}$ increases, the estimation error shows a decreasing trend. One explanation is longer initial trajectory ensures a more accurate initial state estimation.}

\begin{figure}[htb]
    \centering
\includegraphics[width=0.35\textwidth]{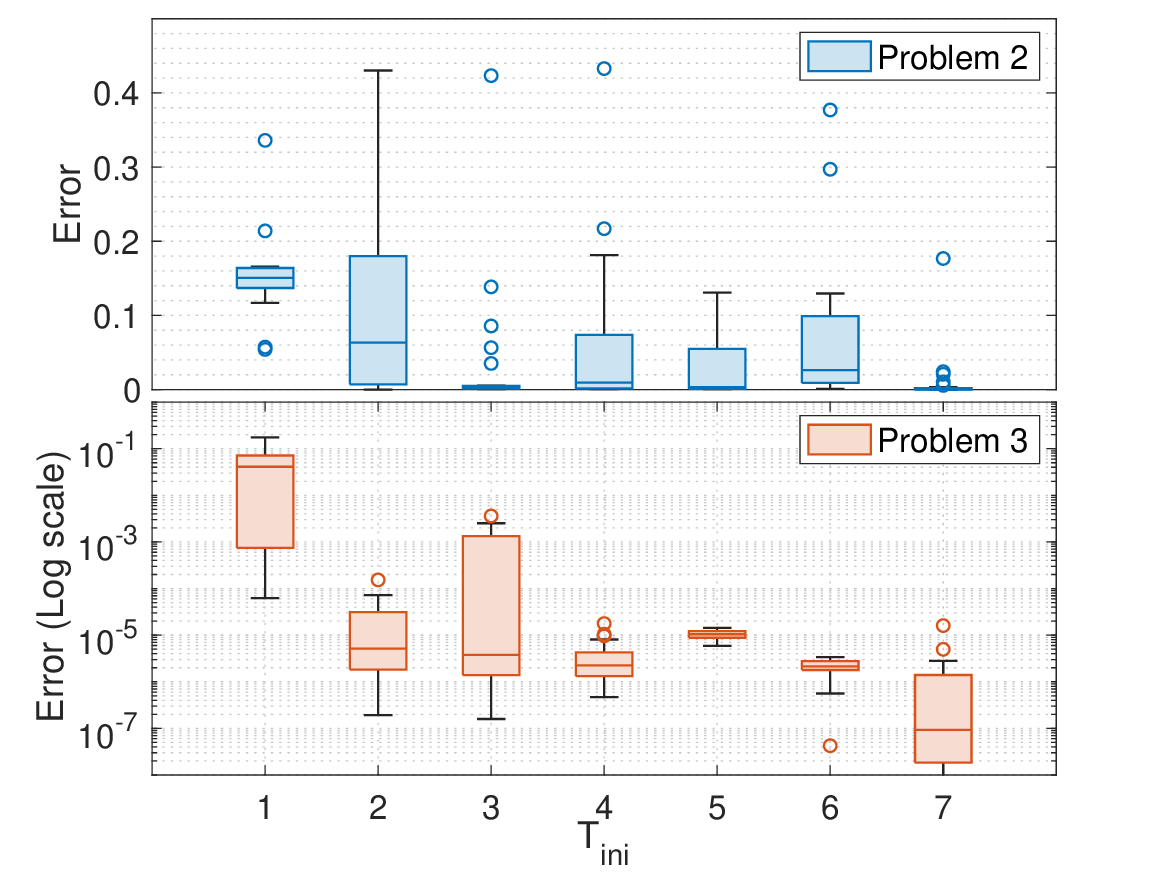}

    \caption{{Estimation errors to $Q,R$ without noises. At each $T_{\textup{ini}}$ we conduct the simulation for $30$ times.}}
    \label{err-compare}

\end{figure}

\begin{table*}[htb]\label{tab:baseline}
    \centering
{
    \caption{Comparison with Baseline Methods}
    \scalebox{0.88}{
    \begin{tabular}{ccccc}
    \toprule
    Method & \makecell{Observation\\model} & \makecell{Data requirement} & \makecell{Error} & \makecell{Computation} \\
    \midrule
    \makecell[l]{Sys-ID IOC \cite{yu2021system}} & \makecell{State}  & \makecell{1 offline traj.~($T\!=\!100$) \\ + 1 optimal traj. \\1 offline traj.~($T\!=\!1000$) \\ + 1 optimal traj.} & \makecell{0.279 \\~\\ 0.035} & \makecell{SysID + Kalman \\+ SVD} \\
    \midrule
    \makecell[l]{MaxEnt IRL \cite{ziebart2008maximum}} & \makecell{State} & \makecell{100 steps \\ 1000 steps} & \makecell{0.334 \\ 0.105} & \makecell{SysID + iterative\\L-BFGS (500+ iters)} \\
    \midrule
    \makecell[l]{\textbf{KKT 3DIOC} \\ \textbf{(Proposed)}} & \makecell{Output \\ ($C\!\neq\!I,D\!\neq\!0$)} & \makecell{1 offline traj.~\textbf{($\mathbf{T\!=\!50}$)} \\ + 1 optimal traj.} & \textbf{0.076} & \makecell{One-shot SDP/QP\\($<$1\,s)} \\
    \bottomrule
    \end{tabular}}
}
\end{table*}

{To demonstrate the efficiency in both computation cost and data requirements, we compare simplified 3DIOC with two representative IOC/IRL baselines:
\begin{enumerate}[label=(\roman*)]
\item \textbf{IOC via system identification approach (Sys-ID IOC) \cite{yu2021system}}. We implement this method following a multi-stage pipeline: (a)~perform subspace system identification (N4SID \cite{van1994n4sid}) for dynamic model $(\hat{A},\hat{B})$; (b)~estimate states via RTS Kalman Smoother; (c)~apply algorithm in \cite{yu2021system} to recover $Q$ and $R$.
\item \textbf{Maximum Entropy IRL (MaxEnt IRL) \cite{ziebart2008maximum}}. This method performs state estimation followed by Maximum Causal Entropy optimization. It fits $Q,R$ by minimizing the negative log-likelihood
\begin{equation}\nonumber
\begin{aligned}
& \mathcal{L}(Q,R) = -\sum_{i=1}^{N} \log P(\tau_i\,|\,Q,R),\\
& P(\tau\,|\,Q,R) \propto \exp\!\Big(-\sum_{t} \tfrac{1}{2}(x_t^\top Q x_t + u_t^\top R u_t)\Big),
\end{aligned}
\end{equation}
where $\tau$ denotes a trajectory and the soft-optimal policy is obtained from the identified model. In practice we optimize $\mathcal{L}$ with L-BFGS-B using Monte Carlo rollouts.
\end{enumerate}
As results shown in Table \ref{tab:baseline}, the proposed KKT-based 3DIOC achieves lower estimation error while requiring less data and computation. This is consistent with our theoretical analysis. The proposed method leverages Fundamental Lemma to establish a \emph{direct} model-free optimality condition, bypassing the two-stage ``identify then optimize'' pipeline that introduces compounding errors in the baselines. Since two baselines require state access, an additional state observer is needed which further degrades the estimation accuracy.
}

We conduct noise-corrupted case on an LQR controller. Suppose $C = I_2, D = 0$ and
\[
A = \begin{pmatrix}
1 & 1 \\ 0  & 1
\end{pmatrix}, B = \begin{pmatrix}
0 \\ 1
\end{pmatrix}, Q = \begin{pmatrix}
1 & 0.2\\0.2 & 0.8
\end{pmatrix}, R =0.4.
\]
Fig. \ref{err-noise} shows the estimation error of Problem \ref{inf_qr_pro}. We add Gaussian noises to the input $\Tilde{u}$. As the signal-to-noise ratio (SNR) (dB) becomes larger, the estimation error and variance both decrease. The estimation error is low for most times, while several large errors appear randomly, which is consistent with our perturbation analysis. The robustness to noise is related to the spectral gap of matrix $\Tilde{\Phi}^p$. To demonstrates the computation efficiency, we report the computation time shown in Table \ref{time-tab}. We can find that though the initial trajectory length becomes longer, the unknown variables number only depends on $n,m$, so the time cost remains a constant. {We also conduct Monte Carlo tests under diverse noise distributions including representative non-Gaussian perturbations: Laplace, bounded uniform and sparse outlier. Fig. \ref{fig:mc_test} shows the qualitative robustness persists across these non-Gaussian noise models.}
\begin{figure}[htb]
\centering
\includegraphics[width=0.35\textwidth]{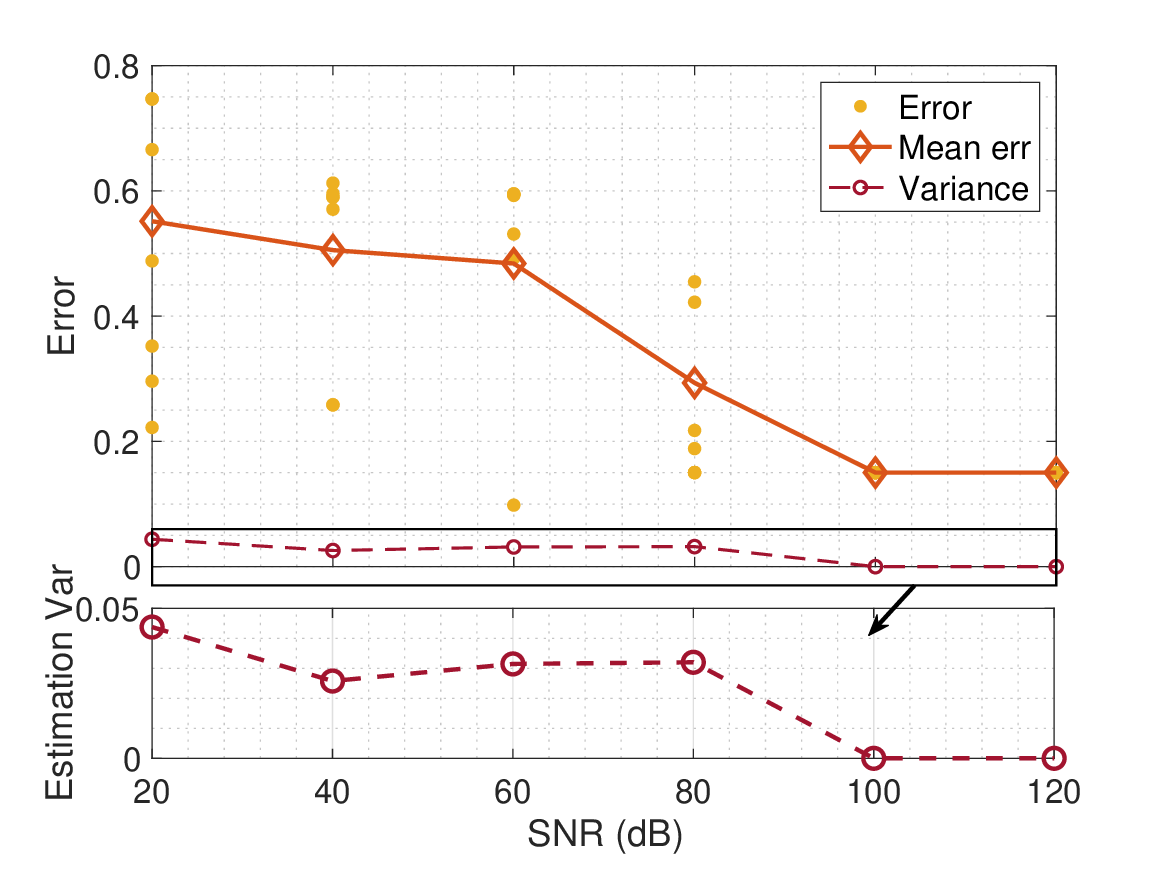}
\caption{Estimation errors with the presence of observation noises. We conduct the simulation $15$ times for each variance. The orange scatters are errors at each experiment. The solid red curve represents mean of the error and the dashed line for variance.}
\label{err-noise}
\end{figure}

\begin{figure}[htb]
{
    \centering
    \includegraphics[width=0.8\linewidth]{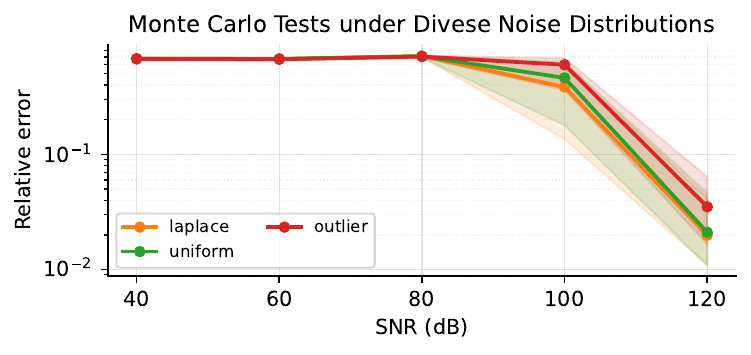}
    \caption{Monte Carlo Validation for noisy case under diverse noise distributions (50 independent trials).}
    \label{fig:mc_test}
}
\end{figure}

\begin{table}[htb]
    \centering
    \caption{Computation Efficiency of Problem \ref{inf_qr_pro}}
	\label{time-tab}
    \scalebox{0.9}{\begin{tabular}{cccccc}
    \toprule    
    $T_{\textup{ini}}$ & 3 & 5 & 7 & 9 & 11 \\    
    \midrule   
{Time(s)} & 0.6879 & 0.5949 & 0.4318 & 0.5959 & 0.4381\\
    \bottomrule   
    \end{tabular}}
\end{table}

\begin{figure}[b]
    \centering
    \includegraphics[width=0.8\linewidth]{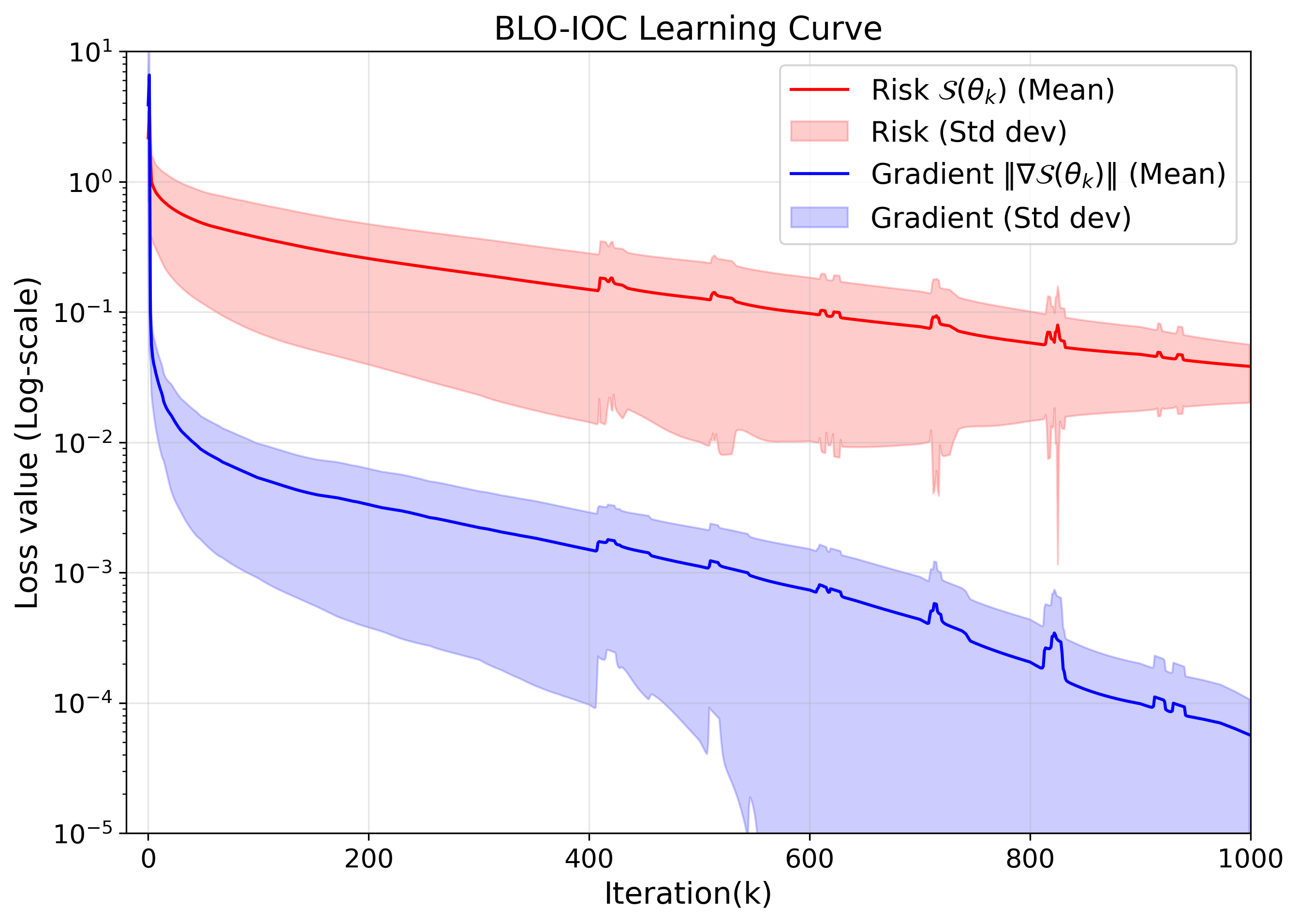}
    \caption{The learning curves for bi-level IOC optimization. The red curve represents the risk loss $\mathcal{S}(\theta_k)$, while the blue curve shows the gradient value $\nabla \mathcal{S}(\theta_k)$ at each iteration.}
    \label{fig:blo-lrc}
\end{figure}
\subsection{Simulation for BLO-IOC}
\begin{figure*}[tb]
    \centering
    \subfigure[Parameter set 1]{
\includegraphics[width=0.3\linewidth]{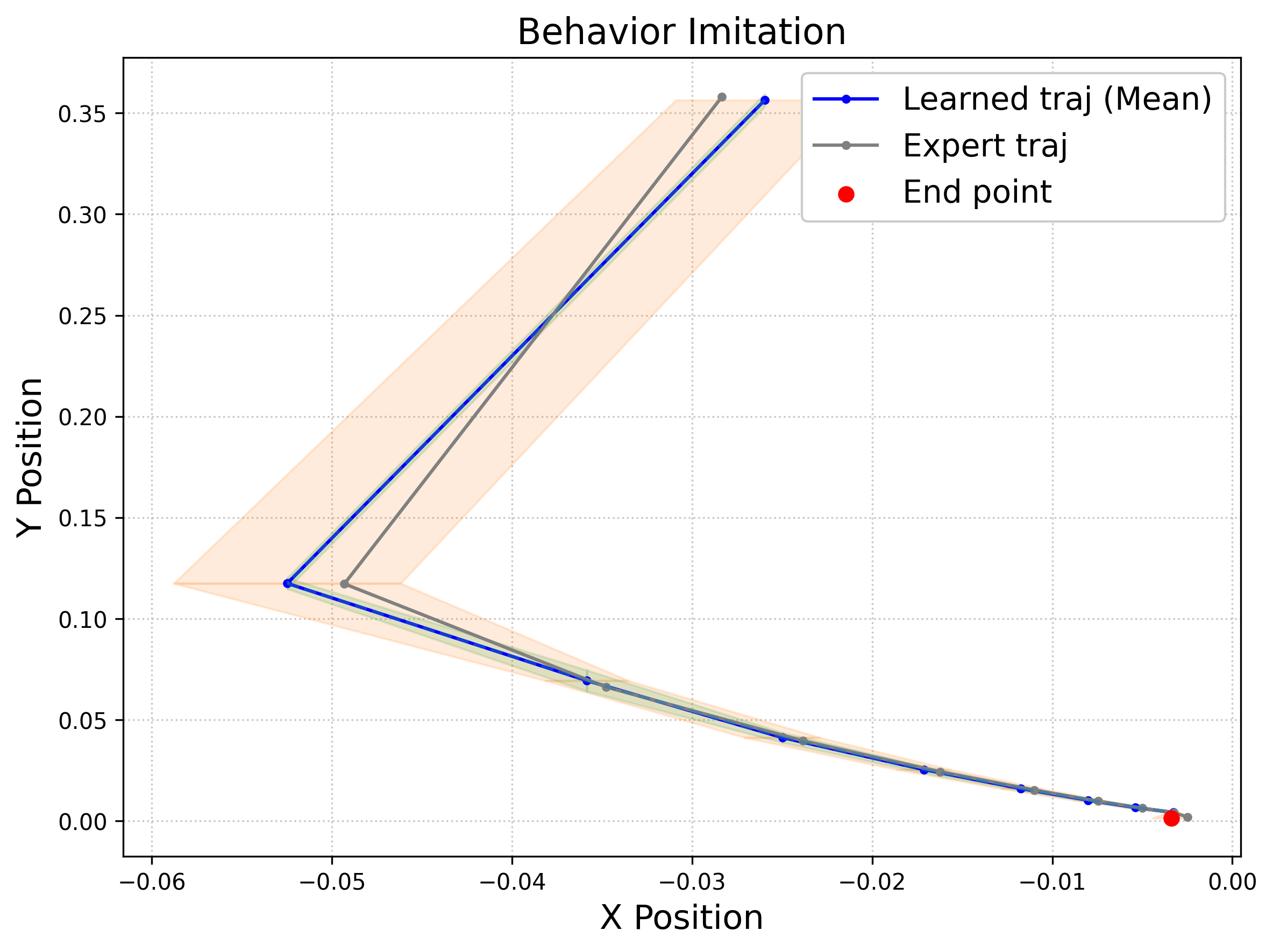}
    }
    \subfigure[Parameter set 2]{
\includegraphics[width=0.3\linewidth]{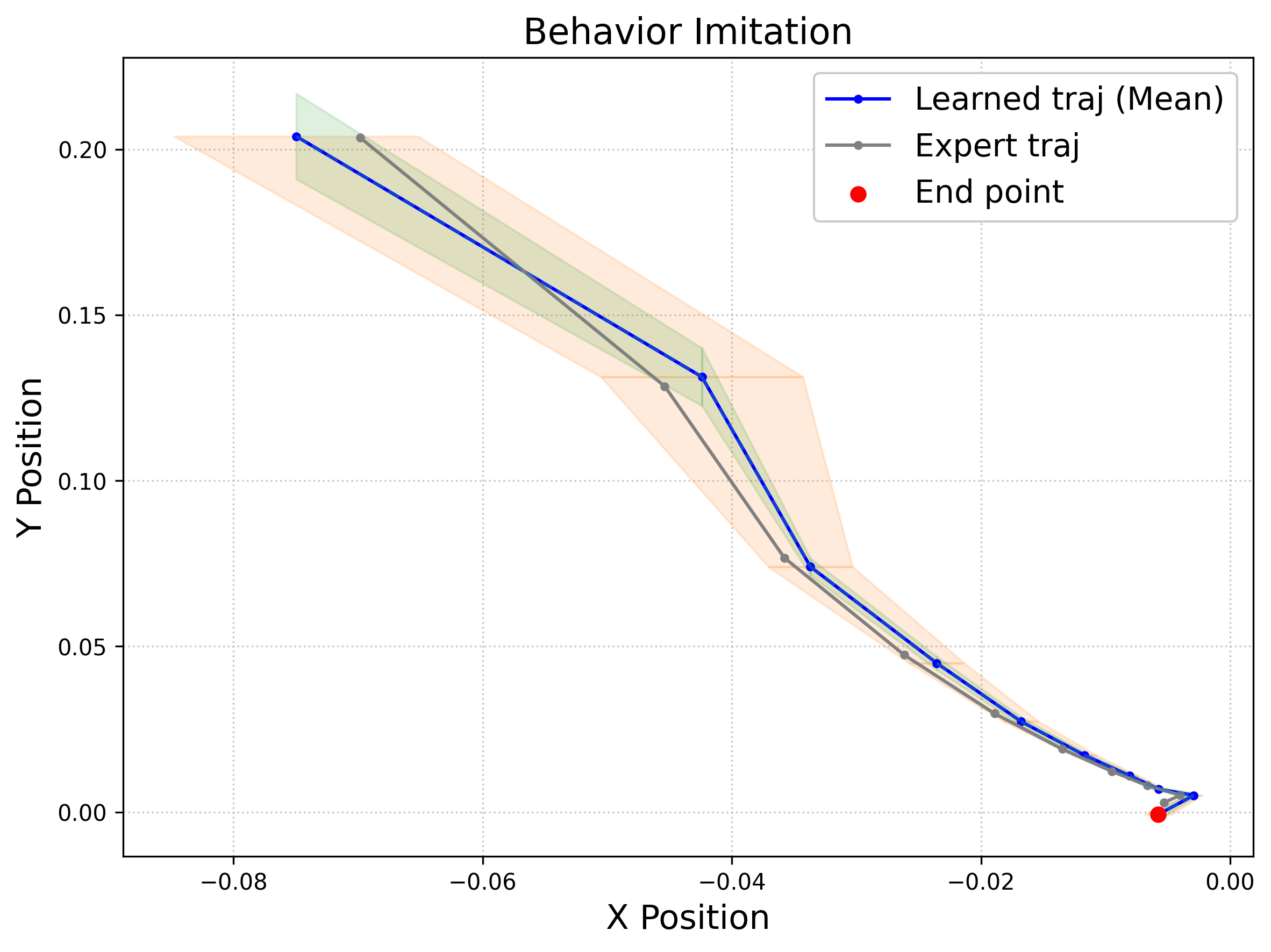}
    }
    \subfigure[Parameter set 2 (with noises)]{
\includegraphics[width=0.3\linewidth]{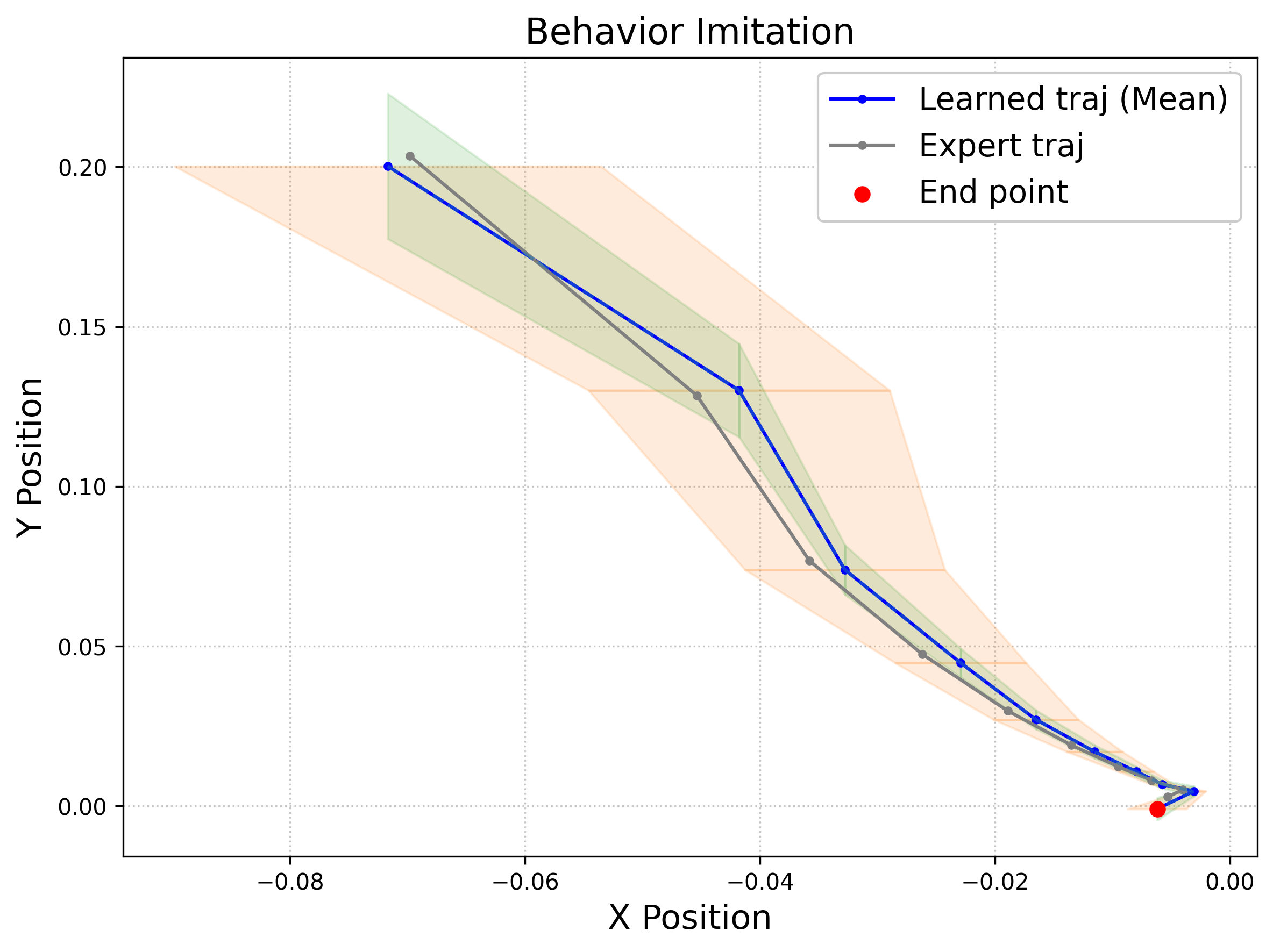}
    }
    \caption{The learned input trajectory generated by the estimation $\hat{\theta}$ and comparison with the expert's behavior. The blue curve is the mean value of learned input trajectories $u^*_{\hat{\theta}}$ and the gray curve is the expert inputs $\Tilde{u}$. The red dot shows the end position. The light shaded areas illustrate the standard deviation. (a) and (b) are conducted under different $Q,R$ sets, while (c) considers observation noises.}
    \label{fig:traj-plot}
\end{figure*}

We then employ the bi-level optimization in Problem \ref{blo-ioc} to estimate $Q,R$. Note that the lower and upper bounds in the constraint set $\mathcal{D}$ are mainly used for convergence analysis. In the practical simulation, we only need to ensure $Q,R$ are positive definite and do not go to infinity (i.e., set $\bar{\epsilon}$ to be large enough) to facilitate the projection. For Algorithm \ref{blo-alg}, let the learning rates $\alpha =1\mathrm{e}{-10}, \beta = 1\mathrm{e}{-9}$. We implement the learning rate for lower level larger than the upper level problem to guarantee the overall convergence \cite{hong2023two} and adjust the rates dynamically based on the gradient descent behavior. The initial guess to the weighting parameter $\theta_0$ is generated from a uniform distribution on $[0,1]^{m(m+1)n(n+1)/2}$. The tolerance decrease sequence can be chosen as $\epsilon_k = \mathcal{O}(k^{-p}), p>1$ and a maximum iteration bound is also set for the lower problem to prevent large computational consuming. Let $\texttt{max\_iter}$ as $1000$ and conduct Algorithm \ref{blo-alg}. The algorithm is run on Intel i9-14900K CPU.

The learning curves during the optimization are shown in Fig. \ref{fig:blo-lrc}. The vertical axis is plotted on a logarithmic scale. We can find as the iteration increases, both the gradient and risk decrease. These illustrate the convergence of the algorithm, and the gap between the generated behavior $u^*_{\theta_k}$ corresponding to learned $\theta_k$ and the target expert's behavior $\Tilde{u}$ gets smaller and smaller. The convergence speed is close to exponential convergence near the local minimum. It is worth noting that though the optimal solution to the inner control problem can be explicitly calculated by \eqref{skkt-eq}, we observe a better numerical stability with gradient decent.

We plot the learned input trajectory corresponding to the optimization result $\hat{\theta}$ in Fig. \ref{fig:traj-plot}. We change the weighting parameters $Q,R$ and imitate two different expert's trajectories. From (a) and (b) we can find the control input generated by the estimation $\hat{\theta}$ is very close to the expert's behavior and imitates the shape, such as trends and turning points, successfully. To see the robustness of the proposed algorithm to noises. We add Gaussian observation noises $\mathcal{N}(0,0.01^2)$ (the SNR is about $20$dB) to the expert's trajectory and the comparison result is shown in (c). The algorithm still imitates the behavior well though the standard deviation becomes larger than noiseless case. 
{Under a conservative dense linear-algebra accounting, the computation cost of a per-iteration in our implementation is dominated by $\mathcal{O}((mN)^3)$, and $\mathcal{O}((mN)^3+(mN)^2 d)$ when accounting for the implicit-differentiation step, with additional matrix multiplications of lower order compared with the cubic solve.}

Our code is available
at: \href{https://github.com/cdqu/3DIOC-Direct-data-driven-inverse-optimal-control}{3dioc-code}.

\section{Conclusion}\label{conc}
This paper proposes a direct data-driven IOC algorithm for LQ control problem in LTI systems. With the input-output representation introduced from the behavioral system theory, the IOC problem is solved with observation trajectories only. Identifiability condition and perturbation analysis of the estimation are provided. We propose both KKT-based and bi-level formulation IOC algorithms, establishing the connection between estimation results and a trade-off between computation cost and robustness. Simulations demonstrate that our algorithm achieves high computation efficiency and requires less data. {By finding an appropriate Koopman operator and a lift mapping \cite{martin2023guarantees}, this 3DIOC framework can also be extended to nonlinear systems, which will be investigated in our future work.}

\appendix
\section{Proof of Proposition \ref{iden_th}}\label{th1-prf}
From equation \eqref{kkt_cons}, we have
\begin{equation}\nonumber
\begin{aligned}
& \mathcal{F}(\lambda_{\textup{ini}}, Q,R) =w_p^T \diag(\lambda_{\textup{ini}}) \mathbf{1}  + 2 U_f^T \mathcal{R} \Tilde{u} + 2 Y_f^T \mathcal{Q} \Tilde{y}\\
& = \begin{pmatrix}
w_p^T & 2 U_f^T & 2 Y_f^T
\end{pmatrix}
\begin{pmatrix}
\diag(\lambda_{\textup{ini}}) &\\ & \diag(\mathcal{R},\mathcal{Q})
\end{pmatrix}
\begin{pmatrix}
\mathbf{1} \\ \Tilde{u}\\\Tilde{y}
\end{pmatrix}.
\end{aligned}
\end{equation}
Vectorize the function. We have
\begin{equation}
\begin{aligned}
& \vecc(\mathcal{F}(\lambda_{\textup{ini}}, Q,R)) =\underbrace{{\begin{pmatrix}
\mathbf{1}^T & \Tilde{u}^T & \Tilde{y}^T
\end{pmatrix}
\otimes  \begin{pmatrix}
w_p^T & 2 U_f^T & 2 Y_f^T
\end{pmatrix}}}_{\Psi} \\
& \cdot {\vecc(\diag(
\diag(\lambda_{\textup{ini}}),\diag(\mathcal{R},\mathcal{Q})
)} = \Psi \cdot \mathcal{M}(\lambda_{\textup{ini}},R,Q).
\end{aligned}
\end{equation}
The size of matrix $\Psi$ is $[T_d-(T_{\textup{ini}}+N) +1] \times (m+p)^2(T_{\textup{ini}}+N)^2$. Since the rank of $\begin{pmatrix}
\mathbf{1}^T & \Tilde{u}^T & \Tilde{y}^T
\end{pmatrix}$ is $1$, considering the property of Kronecker product, we have
\begin{equation}\nonumber
\begin{aligned}
& \rank(\Psi) \!=\! \rank\begin{pmatrix}
w_p^T & 2 U_f^T & 2 Y_f^T
\end{pmatrix} \!=\!\rank({H}_{T_{\textup{ini}}+ N} (w^d))\\
& =\rank(\mathcal{H}_{T_{\textup{ini}}+N} (w^d))= m(T_{\textup{ini}}+N)+n.
\end{aligned}
\end{equation}
There are lots of zeros and duplicate items in the vector $\mathcal{M}(\lambda_{\textup{ini}},R,Q)$. Observing the structure of $\mathcal{M}(\lambda_{\textup{ini}},R,Q)$ and $\Psi$, we can obtain
\begin{equation}\nonumber
\begin{aligned}
\Psi \mathcal{M}(\lambda_{\textup{ini}},R,Q) = \begin{pmatrix}
w_p^T & \Tilde{U}_f & \Tilde{Y}_f\end{pmatrix} \col(\lambda_{\textup{ini}}, \vecc(R), \vecc(Q))
\end{aligned}
\end{equation}
where $\Tilde{U}_f = \sum_{i=0}^{N-1} \Tilde{u}_i^T \otimes 2U_f^T(im+1:(i+1)m,:)$, $\Tilde{Y}_f =  \sum_{i=0}^{N-1} \Tilde{y}_i^T \otimes 2Y_f^T(ip+1:(i+1)p,:)$. {Since weighting matrices $Q,R$ are symmetric, we do further simplification with duplication matrices $D_m, D_p$, which satisfy
\begin{equation}\label{dup-mat}
D_m \vech(R) = \vecc(R), D_p \vech(Q) = \vecc(Q).
\end{equation}
The explicit formula for a duplication matrix for an $m \times m$ matrix is $D_m^T = \sum_{i \geq j} \tau_{ij}\vecc(T_{ij})^T$, where $\tau_{ij}$ is a unit vector of order $\frac{1}{2}m(m+1)$ having value $1$ in the position $(j-1)m+i-\frac{1}{2} j(j-1)$ and $0$ elsewhere, and $T_{ij}$ is an $m\times m$ matrix with $1$ in position $(i,j), (j,i)$ and $0$ elsewhere. }

{With above definitions, denote
\begin{equation}
\Tilde{\Psi} = \begin{pmatrix}
w_p^T & \Tilde{U}_f & \Tilde{Y}_f\end{pmatrix} \cdot \diag(I,D_m,D_p)
\end{equation}
We finally have the equation \eqref{kkt_cons} as
\begin{equation}\label{psi-eq}
\begin{aligned}
& \Psi \cdot \mathcal{M}(\lambda_{\textup{ini}},R,Q) \\&=\begin{pmatrix}
w_p^T & \Tilde{U}_f & \Tilde{Y}_f\end{pmatrix} \col(\lambda_{\textup{ini}}, D_m \cdot \vech(R), D_p \cdot \vech(Q)) \\& = \begin{pmatrix}
w_p^T & \Tilde{U}_f & \Tilde{Y}_f\end{pmatrix} \diag(I,D_m,D_p) \col(\lambda_{\textup{ini}}, \vech(R), \vech(Q))\\&=\Tilde{\Psi} \cdot \col(\lambda_{\textup{ini}}, \vech(R), \vech(Q)) = 0
\end{aligned}
\end{equation}}
Matrix $\Tilde{\Psi}$ is of order $T-(T_{\textup{ini}}+N)+1 \times (m+p)T_{\textup{ini}}+\frac{1}{2}(m^2+m+p^2+p)$. 
Observing \eqref{psi-eq}, if the rank of $\Tilde{\Psi}$
\begin{equation}
\rank(\Tilde{\Psi}) = (m+p)T_{\textup{ini}}+\frac{1}{2}(m^2+m+p^2+p)-1,
\end{equation}
supposing the real weighting matrices are $R,Q$, a solution $(\lambda'_{\textup{ini}}, \vech(R'), \vech(Q'))$ obtained from \eqref{psi-eq} satisfies
$Q = \alpha Q', R = \alpha R'$
for a scalar $\alpha \in \mathbb{R}_{+}$.

\section{Proof of Theorem \ref{s-iden_th}}\label{thm2-prf}
Considering equation \eqref{skkt-eq}, we have
\begin{equation}\nonumber
\begin{aligned}
& \nabla_u J(\Tilde{u}) =: (K_f^T \mathcal{Q} K_f+\mathcal{R}) \Tilde{u} + K_f^T \mathcal{Q}^T K_p z_{\textup{ini}}\\& = \mathcal{R}\Tilde{u} + K_f^T \mathcal{Q} (K_f \Tilde{u} + K_p z_{\textup{ini}}) \\
& = {\begin{pmatrix}
K_f^T & I_{mN}
\end{pmatrix}} \diag(\mathcal{Q}, \mathcal{R}) \begin{pmatrix}
K_f \Tilde{u} + K_p z_{\textup{ini}} \\ \Tilde{u}
\end{pmatrix}=0.
\end{aligned}
\end{equation}
$I_{mN}$ represents the identity matrix of $mN \times mN$. Observing \eqref{y-pseudo}, we can find that $K_f \Tilde{u} + K_p z_{\textup{ini}} := \Tilde{y}^{es}$ is actually an estimation to $\Tilde{y}$.
Let $\Tilde{K}_f = \begin{pmatrix}
K_f^T & I_{mN}
\end{pmatrix}$. We have
\begin{equation}\nonumber
\begin{aligned}
& \vecc(\nabla_u J(\Tilde{u})) = {\begin{pmatrix}
(\Tilde{y}^{es})^T & (\Tilde{u})^T
\end{pmatrix} \otimes \Tilde{K}_f} \cdot \vecc(\diag(\mathcal{Q}, \mathcal{R})).
\end{aligned}
\end{equation}

Similar to the proof of Proposition \ref{iden_th}, we want to eliminate the zeros and duplicate items in $\vecc(\diag(\mathcal{Q},\mathcal{R}))$. We have
\begin{equation}\nonumber
\vecc(\nabla_u J(\Tilde{u})) = \begin{pmatrix}
\Phi_y & \Phi_u
\end{pmatrix} \cdot \col(\vecc(Q),\vecc(R)),
\end{equation}
where $\Phi_y = \sum_{i=0}^{N-1} (\Tilde{y}^{es}_i)^T \otimes K_f^T(ip+1:(i+1)p,:)$ and $\Phi_u = \sum_{i=0}^{N-1} (\Tilde{u}_i)^T \otimes I_{mN}(im+1:(i+1)m,:)$. 
With duplication matrices $D_p,D_m$ defined in \eqref{dup-mat}, denote
\begin{equation}\label{phi}
\Tilde{\Phi} = \begin{pmatrix}
\Phi_y & \Phi_u
\end{pmatrix} \diag(D_p, D_m) = \begin{pmatrix}
    \Phi_y D_p & \Phi_u D_m
\end{pmatrix}.
\end{equation}
We finally derive the constraint as
\begin{equation}\label{fin-s-kkt}
\vecc(\nabla_u J(\Tilde{u})) = \Tilde{\Phi} \cdot \col(\vech(Q),\vech(R))=0.
\end{equation}
Observing \eqref{fin-s-kkt}, the unknown variable is only composed by $Q,R$. If there exists a proper $T_{\textup{ini}}$ making the rank of $\Tilde{\Phi}$
\begin{equation}
\rank(\Tilde{\Phi}) = \frac{1}{2}(m^2+m+p^2+p)-1,
\end{equation}
supposing the real weighting matrices are $Q,R$, a solution $Q',R'$ obtained from \eqref{fin-s-kkt} satisfies
$Q = \alpha Q', R = \alpha R'$
for a scalar $\alpha \in \mathbb{R}_{+}$.

\section{Proof of Corollary \ref{per-ana}}\label{thm3-prf}
{
Notice that in Problem \ref{inf_qr_pro} the input trajectory contains observation noises. According to the derivation in Appendix \ref{thm2-prf}, we have
\begin{align*}
&\Vert (K_f^T \mathcal{Q} K_f+\mathcal{R}) \Tilde{u}^p + K_f^T \mathcal{Q}^T K_p z_{\textup{ini}} \Vert_2^2\\ &= \Vert \Tilde{\Phi}^p \cdot \col(\vech(Q),\vech(R)) \Vert_2^2,
\end{align*}
where $\Tilde{\Phi}^p$ is calculated similarly as $\Tilde{\Phi}$ substituted $\Tilde{u}$ with $\Tilde{u}^p$. Therefore, with the coefficient matrix $\Tilde{\Phi}^p$, Problem \ref{inf_qr_pro} is now a classic Rayleigh quotient problem described as
\[
\min_{\theta} \Vert \Tilde{\Phi}^p \cdot \theta \Vert_2^2 ~~~ s.t. ~ \Vert \theta \Vert_2^2 =1.
\]
Here we substitute $\col(\vech(Q),\vech(R))$ with variable $\theta$. Denote the optimal solution as $\rho(\Tilde{\Phi}^p)$. It is actually the eigenvector of the matrix $(\Tilde{\Phi}^p)^T \Tilde{\Phi}^p$ corresponding to its smallest eigenvalue. 
By carrying out the SVD decomposition, we have
\[
(\Tilde{\Phi}^p)^T \Tilde{\Phi}^p = \begin{pmatrix}
U & u_1
\end{pmatrix} \diag(\Sigma, \sigma_1) \col(V^T, v_1^T),
\]
where $\sigma_1$ is the smallest singular value of $\Tilde{\Phi}^p$ and the optimal solution $\rho(\Tilde{\Phi}^p) = v_1$.
}

{
In this case, considering a perturbation $\Delta$ on the coefficient matrix, the sensitivity of Problem \ref{inf_qr_pro} is measured by the difference between the two smallest eigenvectors corresponding to $\Tilde{\Phi}^p$ and perturbed $\Tilde{\Phi}^p+\Delta$ separately. This can be analyzed by the eigenvalue perturbation theory \cite{rellich1969perturbation}.
Let $\beta$ be the angle between $\rho(\Tilde{\Phi}^p) = v_1$ and $\rho(\Tilde{\Phi}^p + \Delta) = v'_1$. We first show when the perturbation $\Delta$ is relatively small, the Euclidean distance between $v_1$ and $v'_1$ is bounded by $\sin(\beta)$. We have
\[
\Vert v_1-v'_1\Vert_2^2 = \Vert v_1\Vert_2^2 + \Vert v'_1\Vert_2^2 - 2 v_1^T v'_1 = 2-2 \cos(\beta),
\]
where the second equation is due to normalization of the eigenvector (e.g., $\Vert v_1\Vert_2^2 =1$). Using the trigonometric identity, there is
\[
\Vert v_1-v'_1\Vert_2 = \sqrt{2-2 \cos(\beta)} = \sqrt{2} \sqrt{2 \sin^2(\frac{\beta}{2})} = 2 \sin(\frac{\beta}{2}).
\]
Therefore, for small angle $\beta$, we have $\Vert v_1-v'_1\Vert_2 \approx	2 \cdot \frac{\beta}{2} = \beta \approx \sin(\beta)$. Then according to Davis-Kahan-$\sin(\theta)$ theorem, we derive the upper bound
\begin{equation}
\Vert v_1-v'_1\Vert_2 \approx \sin(\beta) \leq \frac{\Vert E \Vert}{|\sigma_1 - \sigma_2|},
\end{equation}
where $\sigma_2$ is the second small singular value of $\Tilde{\Phi}^p$ and
\[
E = (\Tilde{\Phi}^p +\Delta )^T (\Tilde{\Phi}^p +\Delta) - (\Tilde{\Phi}^p)^T \Tilde{\Phi}^p \approx \Delta^T \Tilde{\Phi}^p+ (\Tilde{\Phi}^p)^T \Delta.
\]
we omit the higher order term $\Delta^T \Delta$. The proof is done.
}

\section{Proof of Theorem \ref{blo-conv}}\label{th-3-pf}
Observing Problem \ref{blo-ioc}, we have the following properties:
\begin{itemize}
\item \emph{L-smoothness}: For any $\theta \in \mathcal{D}$, the first derivative of $l$, $\nabla_{u^*_{\theta}} l$, is Lipschitz continuous (w.r.t $u^*_{\theta}$) with constant $L_{l_u} = 2$. For any $\theta \in \mathcal{D}$, the second derivative $\nabla_{u\theta}J$ is Lipschitz continuous (w.r.t $u^*_{\theta}$) with constant $L_{j_{u\theta}}$ and for any $u^*_{\theta} \in \mathbb{R}^{mN}$, $\nabla_{uu}J$ is Lipschitz continuous (w.r.t $\theta$) with constant $L_{j_{uu}}$.
\item \emph{Strong convexity}: For any $\theta \in \mathcal{D}$, $J(u)$ is strongly convex w.r.t $u$ with parameter $\mu_j >0$, i.e., $\mu_j I \preceq \nabla_{uu} J$.
\item \emph{Boundness}: The constraint set $\mathcal{D}$ is convex and compact. Since $\theta$ is bounded, the corresponding behavior $u^*_{\theta}$ is bounded. Then we have $\Vert \nabla_{u^*_{\theta}} l \Vert \leq C_{l_u}$ and $\Vert \nabla_{u\theta}J \Vert \leq C_{j_{u\theta}}$ for some constants $C_{l_u}, C_{j_{u\theta}}>0$.
\end{itemize}
With above properties, the following lemma holds.
\begin{lemma}[Lemma 2.2 in \cite{ghadimi2018approximation}]\label{contin-s}
$\nabla \mathcal{S}$ is Lipschitz continuous in $\theta$ with constant $L_s$, i.e., for any given $\bar{\theta}_1, \bar{\theta}_2 \in \mathcal{D}$, we have
\begin{equation}
\Vert \nabla \mathcal{S}(\bar{\theta}_1) - \nabla \mathcal{S}(\bar{\theta}_2) \Vert \leq L_s \Vert \bar{\theta}_1 - \bar{\theta}_2 \Vert,
\end{equation}
where
\begin{align*}
& L_s = \frac{(L_{l_u} + \bar{C})C_{j_{u \theta}}}{\mu_j} + C_{l_u} (\frac{L_{j_{u \theta}} C_{l_u}}{\mu_j} + \frac{L_{j_{u u}} C_{j_{u\theta}}}{\mu_j^2}),\\
& \bar{C}= \frac{L_{l_u} C_{j_{u \theta}}}{\mu_j} + C_{l_u} (\frac{L_{j_{u \theta}} }{\mu_j} + \frac{L_{j_{u u}} C_{j_{u\theta}}}{\mu_j^2}).
\end{align*}
\end{lemma}
\begin{lemma}
The inner problem optimization in Algorithm \ref{blo-alg} achieves
\[
\Vert \bar{u}_{\theta_k} - u^*_{\theta_k} \Vert \leq \frac{1}{\mu_j} \Vert \nabla_u J(\bar{u}_{\theta_k}; \theta_k) \Vert \leq \frac{1}{\mu_j} \epsilon_k.
\]
\end{lemma}
\noindent This can be directly derived from the strong convexity of the inner problem. We then have $\sum_{k=0}^{\infty} \frac{1}{\mu_j} \epsilon_k < \infty$.

Thus, based on the above analysis, Problem \ref{blo-ioc} satisfies all the (A1-A3) assumptions in \cite{pedregosa2016hyperparameter}. The subsequent proof is similar to Theorem 2 in \cite{pedregosa2016hyperparameter}.

\section{Proof of Corollary \ref{risk-consis}}\label{co-2-pf}
Note that the upper level objective function in Problem \eqref{blo-ioc-noisy} is just the accumulation of that in Problem \ref{blo-ioc} for $\mathcal{T}$ trajectories. Thus Lemma \ref{contin-s} also gives continuity of $\mathcal{S}_\mathcal{T}(\theta)$. By applying the uniform law of large numbers \cite{jennrich1969asymptotic}, we have
\begin{equation}\label{ulln-of-s}
\sup_{\theta \in \mathcal{D}} |\mathcal{S}_\mathcal{T}(\theta) - \mathcal{S}(\theta)| \xrightarrow{P} 0.
\end{equation}
Take any $\theta_0 \in \arg \min \{\mathcal{S}(\theta)|\theta \in \mathcal{D}\}$. There is $\mathcal{S}_\mathcal{T}(\hat{\theta}) \leq \mathcal{S}_\mathcal{T}(\theta_0)$, and we have
\[
\mathcal{S}(\hat{\theta}_\mathcal{T}) + \mathcal{S}_\mathcal{T}(\hat{\theta}_\mathcal{T}) - \mathcal{S}(\hat{\theta}_\mathcal{T}) \leq \mathcal{S}(\theta_0) - \mathcal{S}(\theta_0) + \mathcal{S}_\mathcal{T}(\theta_0).
\]
Rearranging the terms gives
\[
\mathcal{S}(\hat{\theta}_\mathcal{T}) - \mathcal{S}(\theta_0) \leq |\mathcal{S}_\mathcal{T}(\hat{\theta}_\mathcal{T}) - \mathcal{S}(\hat{\theta}_\mathcal{T})|+|\mathcal{S}_\mathcal{T}(\theta_0) - \mathcal{S}(\theta_0)|.
\]
Due to the definition of $\theta_0$, we have $\mathcal{S}(\theta_0) \leq \mathcal{S}(\hat{\theta}_\mathcal{T})$. Holding the continuity of $\mathcal{S}_\mathcal{T}(\theta)$ and \eqref{ulln-of-s}, we derive $\mathcal{S}(\hat{\theta}_\mathcal{T}) - \mathcal{S}(\theta_0) \xrightarrow{P} 0$. The proof is done.

\section{Proof of Lemma \ref{lem-blo-opt}}\label{lem-6-pf}
The lemma is proved by deriving the necessary optimality condition of Problem \ref{blo-ioc}.
Since $u^*_{\theta}$ is also a function with respect to $\theta$, according to the chain rule, the derivative of the upper level objective function to $\theta$ is
\begin{equation}\label{ul-derive}
{\nabla \mathcal{S}} = \frac{\partial l}{\partial u^*_{\theta}} \cdot \nabla_{\theta} u^*_{\theta} = 2(u^*_{\theta}-\Tilde{u})^T\cdot \nabla_{\theta} u^*_{\theta}.
\end{equation}
For $\nabla_{\theta} u^*_{\theta}$, we utilize the implicit function theorem. Calculate the lower level function derivative to $u$. We have
\[
{\nabla_{u}J(u;{\theta})}=2(K_f^T \mathcal{Q} K_f+\mathcal{R}) u + 2 K_f^T \mathcal{Q}^T K_p z_{\textup{ini}}
\]
When the lower level problem achieves the optimality, for any fixed $\theta$, its corresponding $u^*_{\theta}$ satisfies
\[
{\nabla_{u}J(u^*_{\theta};{\theta})=0}.
\]
Then calculate the derivative of both side with respect to $\theta$. We obtain
\[
\nabla_{u u} J(u^*_{\theta}; \theta) \cdot \nabla_{\theta} u^*_{\theta} + \nabla_{u \theta}J(u^*_{\theta};{\theta}) = 0.
\]
Therefore, we have
\begin{align*}
\nabla_{\theta} u^*_{\theta} & = -\{\nabla_{u u} J(u^*_{\theta}; \theta)\}^{-1} \cdot \nabla_{u \theta}J(u^*_{\theta};{\theta})\\
& = -\frac{1}{2} (K_f^T \mathcal{Q} K_f+\mathcal{R})^{-1} \cdot 2\Tilde{\Phi}(u^*_{\theta})
\end{align*}
Note that previous proof (Appendix \ref{thm2-prf}) shows $\nabla_u J(u;\theta)$ can be written as a matrix $\Tilde{\Phi}$ multiplying the parameter $\theta$. Thus we have $\nabla_{u \theta}J(u^*_{\theta};{\theta}) = 2\Tilde{\Phi}(u^*_{\theta})$ here. $\Tilde{\Phi}(u^*_{\theta})$ is defined by substituting the $\Tilde{u}$ by $u^*_{\theta}$ during the calculation.

Substitute the above equation into \eqref{ul-derive}. There is
\begin{equation}
\begin{aligned}
{\nabla \mathcal{S}} = 2(\Tilde{u} - u^*_{\theta})^T \cdot (K_f^T \mathcal{Q} K_f+\mathcal{R})^{-1} \cdot \Tilde{\Phi}(u^*_{\theta}).
\end{aligned}
\end{equation}
Due to the optimality condition of the upper level problem, we have
\[
(\Tilde{u} - u^*_{\theta})^T \cdot (K_f^T \mathcal{Q} K_f+\mathcal{R})^{-1} \cdot \Tilde{\Phi}(u^*_{\theta}) = 0.
\]

\section{Proof of Proposition \ref{connect}}\label{pro-4-pf}
The proof is quite obvious. Suppose we have an solution $Q_1,R_1$ to Problem \ref{sim_3Dioc}. It satisfies the KKT condition \eqref{skkt-eq} which is
\[
(K_f^T \mathcal{Q}_1 K_f+\mathcal{R}_1) \Tilde{u} + K_f^T \mathcal{Q}_1^T K_p z_{\textup{ini}} = 0.
\]
Note that $\diag(Q_1,R_1) \succeq I$. $Q_1,R_1$ are both positive definite matrices. We have $K_f^T \mathcal{Q}_1 K_f+\mathcal{R}_1$ invertible and
\[
\Tilde{u} + (K_f^T \mathcal{Q}_1 K_f+\mathcal{R}_1)^{-1} (K_f^T \mathcal{Q}_1^T K_p z_{\textup{ini}})=0.
\]
Based on Lemma \ref{lem-blo-opt}, by substituting $Q_1,R_1$ into \eqref{blo-opt-con1}, there is
\begin{align*}
\Tilde{u} - u^*_{{\theta}_1} = \Tilde{u} - (-(K_f^T \mathcal{Q}_1 K_f+\mathcal{R}_1)^{-1} (K_f^T \mathcal{Q}_1^T K_p z_{\textup{ini}})) = 0.
\end{align*}
Thus $\nabla \mathcal{S}(\theta_1)=0$ and equation \eqref{blo-opt-con1} is satisfied. For the second-order condition \eqref{blo-opt-con2}, we check the Hessian matrix $\nabla^2 \mathcal{S}$. Denote $P(\theta) = (K_f^T \mathcal{Q} K_f+\mathcal{R})^{-1} \cdot \Tilde{\Phi}(u^*_{\theta})$. We have
\begin{equation}
\begin{aligned}
\nabla^2 \mathcal{S}(\theta_1) &= (\Tilde{u} - u^*_{{\theta_1}}) \nabla_{\theta_1} P(\theta_1) - (\nabla_{\theta_1} u^*_{\theta_1})^T P(\theta_1)\\
& = \underbrace{(\Tilde{u} - u^*_{{\theta_1}})}_{=0} \nabla_{\theta_1} P(\theta_1) + \underbrace{P(\theta_1)^T P(\theta_1)}_{\succ 0} \succ 0.
\end{aligned}
\end{equation}
Thus equation \eqref{blo-opt-con2} is also satisfied.

However, conversely, if we have a solution $Q_2,R_2$ satisfying Lemma \ref{lem-blo-opt}, we can not guarantee the equation $(K_f^T \mathcal{Q}_2 K_f+\mathcal{R}_2) \Tilde{u} + K_f^T \mathcal{Q}_2^T K_p z_{\textup{ini}} = 0$ given the current condition.

{
These examples illustrate practical deployment scenarios under the LTI (or local-LTI) assumption. Extending the proposed direct IOC principle to fully nonlinear settings is an interesting direction for future work.
}

{
\section{Data-Efficiency Comparison with MaxEnt IRL}
\label{app:data_efficiency_maxent}
We provide a more detailed, supplementary data efficiency comparison with MaxEnt IRL \cite{ziebart2008maximum}. See the summarized comparison in main text Table \ref{tab:baseline}.
We use the same finite-horizon LQ simulation setting as in Section 5.1 and vary the available data amount (I/O trajectory $(u,y)$ steps). As shown in Fig.~\ref{fig:eoc_compare_appendix}, the error of proposed KKT-based 3DIOC drops once the data length exceeds the minimal requirement for the Hankel-based construction, and then quickly saturates at a low level; in contrast, the MaxEnt IRL baseline improves more gradually and requires substantially more data (on the order of $10^2$--$10^3$ samples in our setting) to reach a comparable accuracy. This result explicitly validates the data efficiency of our algorithm.
}
\begin{figure}[htbp]
{
    \centering
    \includegraphics[width=0.8\linewidth]{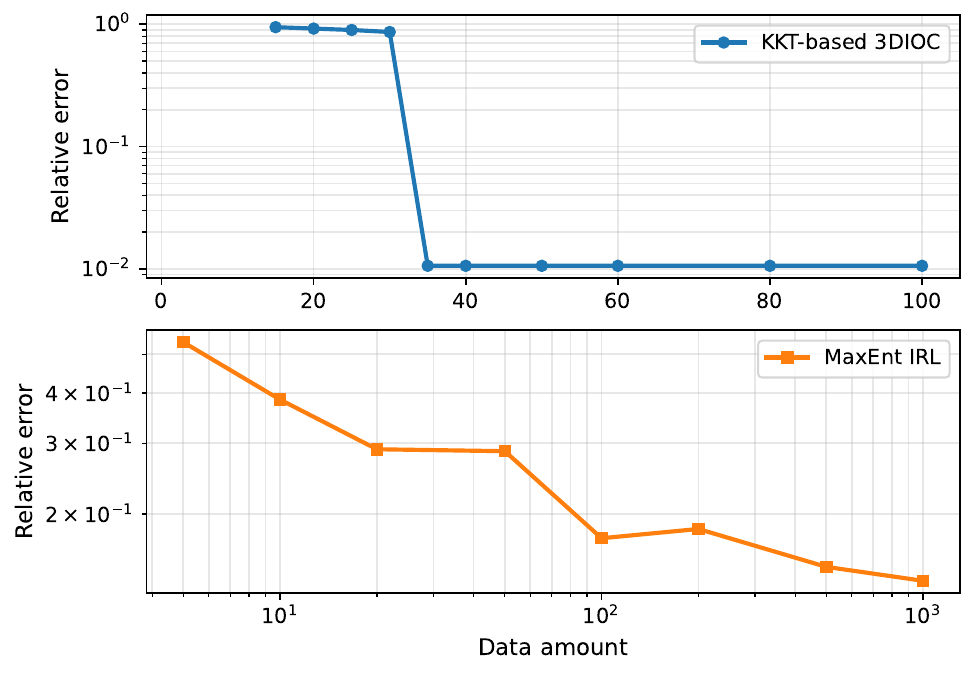}
    \caption{Relative estimation error versus data amount (trajectory steps) for the proposed KKT-based 3DIOC and MaxEnt IRL.}
    \label{fig:eoc_compare_appendix}
}
\end{figure}

\bibliographystyle{unsrt} 
\bibliography{ref}

\begin{thebibliography}{10}

\bibitem{ab2020inverse}
Nematollah Ab~Azar, Aref Shahmansoorian, and Mohsen Davoudi.
\newblock From inverse optimal control to inverse reinforcement learning: A historical review.
\newblock {\em Annual Reviews in Control}, 50:119--138, 2020.

\bibitem{ravichandar2020recent}
Harish Ravichandar, Athanasios~S Polydoros, Sonia Chernova, and Aude Billard.
\newblock Recent advances in robot learning from demonstration.
\newblock {\em Annual Review of Control, Robotics, and Autonomous Systems}, 3:297--330, 2020.

\bibitem{xu2022energy}
Yifei Xu, Jianwen Xie, Tianyang Zhao, Chris Baker, Yibiao Zhao, and Ying~Nian Wu.
\newblock Energy-based continuous inverse optimal control.
\newblock {\em IEEE transactions on neural networks and learning systems}, 34(12):10563--10577, 2022.

\bibitem{mainprice2016goal}
Jim Mainprice, Rafi Hayne, and Dmitry Berenson.
\newblock Goal set inverse optimal control and iterative replanning for predicting human reaching motions in shared workspaces.
\newblock {\em IEEE Transactions on Robotics}, 32(4):897--908, 2016.

\bibitem{ruan2023causal}
Kangrui Ruan, Junzhe Zhang, Xuan Di, and Elias Bareinboim.
\newblock Causal imitation learning via inverse reinforcement learning.
\newblock In {\em The Eleventh International Conference on Learning Representations}, 2023.

\bibitem{chan2025inverse}
Timothy~CY Chan, Rafid Mahmood, and Ian~Yihang Zhu.
\newblock Inverse optimization: Theory and applications.
\newblock {\em Operations Research}, 73(2):1046--1074, 2025.

\bibitem{priess2014solutions}
M~Cody Priess, Richard Conway, Jongeun Choi, John~M Popovich, and Clark Radcliffe.
\newblock Solutions to the inverse lqr problem with application to biological systems analysis.
\newblock {\em IEEE Transactions on Control Systems Technology}, 23(2):770--777, 2014.

\bibitem{yu2021system}
Chengpu Yu, Yao Li, Hao Fang, and Jie Chen.
\newblock System identification approach for inverse optimal control of finite-horizon linear quadratic regulators.
\newblock {\em Automatica}, 129:109636, 2021.

\bibitem{qu2023control}
Chendi Qu, Jianping He, Xiaoming Duan, and Shukun Wu.
\newblock Control input inference of mobile agents under unknown objective.
\newblock {\em IFAC-PapersOnLine}, 56, 2023.

\bibitem{zhang2024statistically}
Han Zhang and Axel Ringh.
\newblock Statistically consistent inverse optimal control for discrete-time indefinite linear--quadratic systems.
\newblock {\em Automatica}, 166:111705, 2024.

\bibitem{jin2022learning}
Wanxin Jin, Todd~D Murphey, Dana Kuli{\'c}, Neta Ezer, and Shaoshuai Mou.
\newblock Learning from sparse demonstrations.
\newblock {\em IEEE Transactions on Robotics}, 39(1):645--664, 2022.

\bibitem{baimukashev2024automated}
Daulet Baimukashev, Gokhan Alcan, and Ville Kyrki.
\newblock Automated feature selection for inverse reinforcement learning.
\newblock {\em arXiv preprint arXiv:2403.15079}, 2024.

\bibitem{ljung2010perspectives}
Lennart Ljung.
\newblock Perspectives on system identification.
\newblock {\em Annual Reviews in Control}, 34(1):1--12, 2010.

\bibitem{byeon2022inverse}
Sooyung Byeon, Dawei Sun, and Inseok Hwang.
\newblock An inverse optimal control approach for learning and reproducing under uncertainties.
\newblock {\em IEEE Control Systems Letters}, 7:787--792, 2022.

\bibitem{verheijen2023handbook}
PCN Verheijen, Valentina Breschi, and Mircea Lazar.
\newblock Handbook of linear data-driven predictive control: Theory, implementation and design.
\newblock {\em Annual Reviews in Control}, 56:100914, 2023.

\bibitem{berberich2020data}
Julian Berberich, Johannes K{\"o}hler, Matthias~A M{\"u}ller, and Frank Allg{\"o}wer.
\newblock Data-driven model predictive control with stability and robustness guarantees.
\newblock {\em IEEE Transactions on Automatic Control}, 66(4):1702--1717, 2020.

\bibitem{de2019formulas}
Claudio De~Persis and Pietro Tesi.
\newblock Formulas for data-driven control: Stabilization, optimality, and robustness.
\newblock {\em IEEE Transactions on Automatic Control}, 65(3):909--924, 2019.

\bibitem{ziebart2008maximum}
Brian~D Ziebart, Andrew~L Maas, J~Andrew Bagnell, Anind~K Dey, et~al.
\newblock Maximum entropy inverse reinforcement learning.
\newblock In {\em Aaai}, volume~8, pages 1433--1438. Chicago, IL, USA, 2008.

\bibitem{xue2021inverse}
Wenqian Xue, Patrik Kolaric, Jialu Fan, Bosen Lian, Tianyou Chai, and Frank~L Lewis.
\newblock Inverse reinforcement learning in tracking control based on inverse optimal control.
\newblock {\em IEEE Transactions on Cybernetics}, 52(10), 2021.

\bibitem{garrabe2023convex}
Emiland Garrabe, Hozefa Jesawada, Carmen Del~Vecchio, and Giovanni Russo.
\newblock On convex data-driven inverse optimal control for nonlinear, non-stationary and stochastic systems.
\newblock {\em arXiv preprint arXiv:2306.13928}, 2023.

\bibitem{garrabe2025convex}
Emiland Garrabe, Hozefa Jesawada, Carmen Del~Vecchio, and Giovanni Russo.
\newblock On convex data-driven inverse optimal control for nonlinear, non-stationary and stochastic systems.
\newblock {\em Automatica}, 173:112015, 2025.

\bibitem{guo2023imitation}
{Guo, Taosha and Al Makdah, Abed AlRahman and Krishnan, Vishaal and Pasqualetti, Fabio}.
\newblock {Imitation and transfer learning for LQG control}.
\newblock {\em {IEEE Control Systems Letters}}, {7}:{2149--2154}, {2023}.

\bibitem{donge2022multiagent}
{Donge, Vrushabh S and Lian, Bosen and Lewis, Frank L and Davoudi, Ali}.
\newblock {Multiagent Graphical Games With Inverse Reinforcement Learning}.
\newblock {\em {IEEE Transactions on Control of Network Systems}}, {10}({2}):{841--852}, {2022}.

\bibitem{donge2024efficient}
{Donge, Vrushabh S and Lian, Bosen and Lewis, Frank L and Davoudi, Ali}.
\newblock {Efficient Reward-Shaping for Multiagent Systems}.
\newblock {\em {IEEE Transactions on Control of Network Systems}}, {2024}.

\bibitem{liang2023data}
Zihao Liang, Wenjian Hao, and Shaoshuai Mou.
\newblock A data-driven approach for inverse optimal control.
\newblock In {\em 2023 62nd IEEE Conference on Decision and Control (CDC)}, pages 3632--3637. IEEE, 2023.

\bibitem{markovsky2021behavioral}
Ivan Markovsky and Florian D{\"o}rfler.
\newblock Behavioral systems theory in data-driven analysis, signal processing, and control.
\newblock {\em Annual Reviews in Control}, 52:42--64, 2021.

\bibitem{willems2005note}
Jan~C Willems, Paolo Rapisarda, Ivan Markovsky, and Bart~LM De~Moor.
\newblock A note on persistency of excitation.
\newblock {\em Systems \& Control Letters}, 54(4):325--329, 2005.

\bibitem{markovsky2022identifiability}
Ivan Markovsky and Florian D{\"o}rfler.
\newblock Identifiability in the behavioral setting.
\newblock {\em IEEE Transactions on Automatic Control}, 68:1667--1677, 2022.

\bibitem{markovsky2008data}
Ivan Markovsky and Paolo Rapisarda.
\newblock Data-driven simulation and control.
\newblock {\em International Journal of Control}, 81(12):1946--1959, 2008.

\bibitem{parikh2014proximal}
Neal Parikh, Stephen Boyd, et~al.
\newblock Proximal algorithms.
\newblock {\em Foundations and trends{\textregistered} in Optimization}, 1(3):127--239, 2014.

\bibitem{aswani2018inverse}
Anil Aswani, Zuo-Jun Shen, and Auyon Siddiq.
\newblock Inverse optimization with noisy data.
\newblock {\em Operations Research}, 66(3):870--892, 2018.

\bibitem{van1994n4sid}
Peter Van~Overschee and Bart De~Moor.
\newblock N4sid: Subspace algorithms for the identification of combined deterministic-stochastic systems.
\newblock {\em Automatica}, 30(1):75--93, 1994.

\bibitem{hong2023two}
Mingyi Hong, Hoi-To Wai, Zhaoran Wang, and Zhuoran Yang.
\newblock A two-timescale stochastic algorithm framework for bilevel optimization: Complexity analysis and application to actor-critic.
\newblock {\em SIAM Journal on Optimization}, 33(1):147--180, 2023.

\bibitem{martin2023guarantees}
{Martin, Tim and Sch{\"o}n, Thomas B and Allg{\"o}wer, Frank}.
\newblock {Guarantees for data-driven control of nonlinear systems using semidefinite programming: A survey}.
\newblock {\em {Annual Reviews in Control}}, page {100911}, {2023}.

\bibitem{rellich1969perturbation}
Franz Rellich.
\newblock {\em Perturbation theory of eigenvalue problems}.
\newblock CRC Press, 1969.

\bibitem{ghadimi2018approximation}
Saeed Ghadimi and Mengdi Wang.
\newblock Approximation methods for bilevel programming.
\newblock {\em arXiv preprint arXiv:1802.02246}, 2018.

\bibitem{pedregosa2016hyperparameter}
Fabian Pedregosa.
\newblock Hyperparameter optimization with approximate gradient.
\newblock In {\em International conference on machine learning}, pages 737--746. PMLR, 2016.

\bibitem{jennrich1969asymptotic}
Robert~I Jennrich.
\newblock Asymptotic properties of non-linear least squares estimators.
\newblock {\em The Annals of Mathematical Statistics}, 40(2):633--643, 1969.

\end{thebibliography}

\textbf{Chendi Qu} received the B.E. degree in the Department of Automation from Tsinghua University, Beijing, China, in 2021. 
She is currently working toward the Ph.D. degree with the Department of Automation, Shanghai Jiao Tong University, Shanghai, China. 
She is a member of Intelligent Wireless Networks and Cooperative Control group. 
Her research interest lies in the intersection of optimal control and optimization. She currently focuses on learning from demonstrations (inverse optimal control) applied for robots.

\textbf{Jianping He} 
(SM’19) is an Associate Professor in the Department of
Automation at Shanghai Jiao Tong University. He received the Ph.D. degree in control science and engineering from Zhejiang University, Hangzhou, China, in 2013, and had been a research fellow in the Department of Electrical and Computer Engineering at University of Victoria, Canada, from Dec. 2013 to Mar. 2017. His research interests mainly include the distributed learning,
control and optimization, security and privacy in network systems.

Dr. He serves as an Associate Editor for IEEE Trans. Control of Network Systems, IEEE Open Journal of Vehicular Technology, and KSII Trans. Internet and Information Systems. He was also a Guest Editor of IEEE TAC, IEEE TII, International Journal of Robust and Nonlinear Control, etc. He was the winner of Outstanding Thesis Award, Chinese Association of Automation, 2015. He received the best paper award from IEEE WCSP'17, the best conference paper
award from IEEE PESGM'17, and was a finalist for the best student paper award from IEEE ICCA'17, and the finalist best conference paper award from IEEE VTC'20-FALL.

\textbf{Xiaoming Duan} 
is an assistant professor in the Department of Automation at Shanghai Jiao Tong University.
He obtained his B.E. degree in Automation from the
Beijing Institute of Technology in 2013, his Master’s Degree in Control Science and Engineering from Zhejiang
University in 2016, and his Ph.D. degree in Mechanical
Engineering from the University of California at Santa
Barbara in 2020. He was a postdoctoral fellow in
the Oden Institute for Computational Engineering and
Sciences at the University of Texas at Austin in 2021.
His research interests include robotics, multi-agent
systems, and autonomous systems.

\end{document}